%% file: paper.tex
\documentclass[10pt]{article}
\usepackage{amsmath,amssymb}
\usepackage{latexsym}
\usepackage{bm}
\usepackage{graphicx}
\input{macros.tex}

\newtheorem{thm}{Theorem}[section]

\newtheorem{lem}[thm]{Lemma}

\newtheorem{defn}{Definition}[section]

\newtheorem{rem}{Remark}[section]

\def\IR{\mbox{\rm I\hspace{-.15em}R}}

\title{Functionally-fitted explicit pseudo two-step Runge-Kutta-Nystr\"{o}m methods}

\author{N. S. Hoang\footnotemark[1] 
}
\date{}

\begin{document}
\maketitle
\renewcommand{\thefootnote}{\fnsymbol{footnote}}
\footnotetext[1]{
Department of Mathematics, University of West Georgia, nhoang@westga.edu}

\begin{abstract}
\noindent
A general class of functionally-fitted explicit pseudo two-step Runge-Kutta-Nystr\"{o}m (FEPTRKN) methods for solving second-order initial value problems
has been studied. These methods can be considered generalized explicit pseudo two-step Runge-Kutta-Nystr\"{o}m (EPTRKN) methods. 
We proved that an $s$-stage FEPTRKN method has step order $p=s$ and stage order $r=s$ for any set of distinct collocation parameters $(c_i)_{i=1}^s$. Supperconvergence for the accuracy orders of these methods can be obtained if the collocation parameters $(c_i)_{i=1}^s$ satisfy some orthogonality conditions. We proved that an $s$-stage FEPTRKN method can attain accuracy order $p=s+3$. Numerical experiments have shown that the new FEPTRKN methods work better than do EPTRKN methods on problems whose solutions can be well approximated by the functions in bases on which these FEPTRKN methods are developed. 

{\bf Keywords.}
functionally-fitted, generalized collocation, variable coefficients, 
two-step explicit RKN, nonstiff ODEs.
\end{abstract}

\section{Introduction}

Consider the initial value problem
\begin{equation}
\label{sys}
y''(t) = f(t,y(t)),\quad
y(0) = y_0,\quad y'(0) = y'_0,\quad
t\in [t_0,\, t_0+T]
\end{equation}
where $f:[t_0,t_0+T]\times \IR\to \IR$ and $f(t,y)$ is continuous with respect to $t$ and satisfies a 
Lipschitz condition with respect to $y$.
For simplicity of presentation we state equation \eqref{sys} in scalar form. However, 
with a change of notation, the discussion also holds when equation \eqref{sys} is in vector form. One of the approaches to solve numerically equation \eqref{sys} is to rewrite the 
equation as a system of first-order ODEs and use numerical methods to solve this system. There are many methods for solving systems of first-order ODEs (see, e.g., \cite{Burrage}, \cite{Hairer}). 
The disadvantage of this approach is that the sizes of the obtained systems are twice as large as the sizes of the orignal systems. 
Also, this approach does not take advantage of the special form of equation \eqref{sys}. 
It has been of interest to develop methods to solve equation \eqref{sys} without rewriting it as a system of first-order ODEs. 

Numerical methods for solving equation \eqref{sys} directly have been developed extensively in the literature (see, e.g., \cite{Coleman}, \cite{cong99},\cite{cong00}, \cite{cong01}, and \cite{cong91}). Among direct methods for solving \eqref{sys}, Runge-Kutta-Nystr\"{o}m (RKN) methods are preferable.  
An $s$-stage RKN method is defined by its Butcher-tableau as follows
$$
\renewcommand{\arraystretch}{1.2}
\begin{array}{c|c}
\bm{c} & \bm{A} \\ \hline
        & \bm{b}^T\\
        & \bm{d}^T
\end{array},~~
\bm{A} = [a_{ij}] \in \IR^{s\times s},~~ 
\bm{b} = (b_1, ..., b_s)^T,~~ 
\bm{d} = (d_1, ..., d_s)^T,~~
\bm{c} = (c_1, ..., c_s)^T. 
$$ 
Given $y_{n}$ and $y_n'$, the approximations of $y(t_n)$ and $y'(t_n)$ at the $n$-th step, 
the approximations of $y(t_{n+1})$ and $y'(t_{n+1})$ at the $(n+1)$-th step 
defined by the $s$-stage RKN method $(\bm{c},\bm{A},\bm{b},\bm{d})$ are
\begin{align}
y_{n+1} &= y_n + hy'_n +
h^2\sum_{j=1}^{s}b_jf(t_n+c_jh,Y_{n,j}), \label{RKNiterate}\\
y'_{n+1} &= y'_n + 
h\sum_{j=1}^{s}d_jf(t_n+c_jh,Y_{n,j}),\\
Y_{n,i} &= y_n + c_ihy'_n+
h^2\sum_{j=1}^{s}a_{ij}f(t_n+c_jh,Y_{n,j}),\qquad i=1, ..., s\label{RKNstage}.
\end{align}

For an implicit RKN method, the matrix $\bm{A}$ is nonsingular and 
\eqref{RKNstage}
is a system of nonlinear equations. In this case it requires special techniques such as 
Newton methods or fixed-point iterations to solve equation \eqref{RKNstage}.
For an explicit RKN method, $\bm{A}$ is strictly 
lower-triangular and $c_1=0$. 
The stage values $(Y_{n,i})_{i=1}^s$ can be easily computed from
\begin{equation}
\label{explicitRKstage}
Y_{n,1} = y_n,\quad Y_{n,i} = y_n + hc_iy'_n +
h^2\sum_{j=1}^{i-1}a_{ij}f(t_n+c_jh,Y_{n,j}),\quad i=2, ..., s.
\end{equation}

Recently, much work has been devoted to the study of functionally-fitted methods. These methods are developed to integrate an ODE exactly if the solution of the ODE is a linear combination of some certain basis functions (see, e.g., \cite{Coleman}, \cite{Ambrosio}, \cite{Franco2}, \cite{nguyen2}, \cite{nguyen3}, \cite{nguyenfeptrk}, \cite{simos}). For example, trigonometric methods have been developed to solve periodic or nearly periodic problems (see, e.g., \cite{Ambrosio}, \cite{nguyen1}, \cite{simos}). Numerical experiments have shown that trigonometrically-fitted methods 
are superior to classical Runge-Kutta methods for solving ODEs whose solutions are periodic or nearly periodic functions with known frequencies (see, e.g., \cite{Franco2}, \cite{nguyen1}, \cite{nguyen2}, \cite{nguyen3}, \cite{nguyenfeptrk}, \cite{Ozawa2}, \cite{simos}). 

In \cite{nguyenfeptrk}, a general class of functionally-fitted explicit pseudo two-step Runge-Kutta (FEPTRK) methods was developed. These methods can be considered generalized explicit pseudo two-step Runge-Kutta (EPTRK) methods. 
Although EPTRK methods were originally developed for parallel computers (cf. \cite{congeptrk99}), it was shown in \cite{nguyenfeptrk} that (F)EPTRK methods work better than one of the most frequently used Runge-Kutta methods the DOPRI45 method even in sequential computing environments on non-stiff problems. 
In \cite{cong99}, a class of explicit pseudo two-step Runge-Kutta-Nystr\"{o}m (EPTRKN) methods for solving non-stiff problem \eqref{sys} was developed. 
In particular, it was proved in \cite{cong99} that an $s$-stage EPTRKN method has step order $p=s$ for any set of collocation points $(c_i)_{i=1}^s$. 
For super-convergence it was shown in \cite{cong99} that an $s$-stage EPTRKN can attain step accuracy order up to $p=s+2$ if the collocation parameters $(c_i)_{i=1}^s$ satisfy some orthogonality conditions. The accuracy study in \cite{cong99} was done by using Taylor series expansion technique. 
It was shown in \cite{cong00} that these methods were more efficient than some of the most commonly used methods for solving non-stiff large-scale systems of ODEs. When solving non-stiff problems it is the accuracy not the stability that controls the stepsizes of numerical methods. 

In this paper, we study a general class of functionally-fitted explicit pseudo Runge-Kutta-Nystr\"{o}m (FEPTRKN) methods. 
Similar to FEPTRK methods in \cite{nguyenfeptrk}, these methods have the advantage of integrating an ODE exactly if the solution of the ODE is a linear combination of functions in bases on which the FEPTRKN methods are developed. Using a similar collocation framework as in \cite{nguyenfeptrk}, we have shown in this paper that an $s$-stage FEPTRKN method has stage order $p=s$ for any set of collocation parameters $(c_i)_{i=1}^s$. 
When the basis functions are polynomials, we recover the results for EPTRKN methods in \cite{cong99}. 
Moreover, we proved in this paper that an $s$-stage FEPTRKN can attain step order $p=s+3$. A consequence of this result is that an $s$-stage EPTRKN method can have step order $p=s+3$. This superconvergence result is new and allows us to construct methods of higher accuracy orders than those in \cite{cong99,cong01} for a given number of stage $s$. 
Using this superconvergence result, we have constructed $4$-stage, $5$-stage, and $6$-stage FEPTRKN methods having accuracy order $7$, $8$, and $9$, respectively. 
Numerical experiments with the new methods have shown that these methods attain accuracy orders as indicated by the theory. Moreover, these FEPTRKN methods have performed much better than the corresponding EPTRKN methods when the basis functions are suitably chosen. 

\section{Explicit pseudo two-step RKN (EPTRKN) methods}
\label{sec:eptrk}

The iteration scheme
of a conventional RKN method based on
\eqref{RKNiterate}--\eqref{RKNstage} can be represented
compactly as
\begin{subequations}
\begin{align}
\label{eq6a}
\bm{Y}_n &= \bm{e}y_n+h\bm{c}y'_n+h^2\bm{A}f(\bm{e}t_n+\bm{c}h,\bm{Y}_n) \in \IR^s,\\
y_{n+1} &= y_n + hy'_n+ h^2\bm{b}^Tf(\bm{e}t_n+\bm{c}h,\bm{Y}_n) \in \IR,\\
y'_{n+1} &= y'_n + h\bm{d}^Tf(\bm{e}t_n+\bm{c}h,\bm{Y}_n) \in \IR,
\end{align}
\end{subequations}
where 
$\bm{Y}_n:=(Y_{n,1},...,Y_{n,s})^T$ and 
$f(\bm{e}t_n+\bm{c}h,\bm{Y}_n) := 
(f(t_n+c_1h,Y_1), ..., f(t_n+c_sh,Y_s))^T$. For implicit RKN methods, one has to solve nonlinear equation \eqref{eq6a} for the stage vector $\bm{Y}_n$ and this requires extra computation time. It should be noted that all classical collocation RKN methods are implicit (cf. \cite{cong91}). It is well known that implicit RK(N) methods should only be used for solving stiff problems. 
For non-stiff problems, explicit methods are computationally cheaper as the stage vector $\bm{Y}_n$ can be computed easily.

In \cite{cong99}  Cong defined 
the iteration scheme of an $s$-stage explicit pseudo two-step RKN (EPTRKN) method as 
\begin{subequations}
\label{EPTRKN}
\begin{align}
y_{n+1} &= y_{n} + hy'_n+ h^2\bm{b}^Tf(\bm{e}t_{n}+\bm{c}h,\bm{Y}_{n}) \in \IR,\\
y'_{n+1} &= y'_{n} + h\bm{d}^Tf(\bm{e}t_{n}+\bm{c}h,\bm{Y}_{n}) \in \IR,\\
\label{EPTRKNc}
\bm{Y}_{n+1} &= \bm{e}y_{n+1} + h\bm{c}y'_{n+1}+h^2\bm{A}f(\bm{e}t_{n}+\bm{c}h,\bm{Y}_{n}) \in \IR^s,
\end{align}
\end{subequations}
where, again,  
$y_n\approx y(t_n)$, $y'_n\approx y'(t_n)$, and 
$\bm{Y}_n=(Y_{n,1},...,Y_{n,s})^T\approx y(\bm{e}t_n+\bm{c}h)=(y(t_n+c_1h),...,y(t_n+c_sh))^T$. 
The advantage of EPTRKN methods over implicit RKN methods is: the stage vector $\bm{Y}_{n+1}$ in equation \eqref{EPTRKNc} can be computed explicitly using the values of $y_n$, $y_n'$, and $\bm{Y}_{n}$ from the previous step. 
The scheme needs $s$ sufficiently accurate starting values 
to define $\bm{Y}_0$, but these can be obtained by a conventional method. 
Since the components of $f(\bm{e}t_{n}+\bm{c}h,\bm{Y}_{n})$ can be
evaluated independently in parallel computing environments, EPTRKN methods are ideally suited 
for parallel computers. Note that on parallel computing environments, EPTRKN methods use only one function evaluation of $f(t_{n}+c_ih,Y_{n,i})$ per step. 

\section{Functionally-fitted explicit pseudo two-step RKN (FEPTRKN) methods}
\label{sec:feptrkn}

Let us first give a definition of functionally-fitted explicit pseudo two-step RKN (FEPTRKN) methods. 
This requires choosing a set of basis functions 
$\{u_i(t)\}_{i=1}^{s}$ that are sufficiently smooth and linearly 
independent and satisfy the integration scheme 
\eqref{EPTRKN} exactly.

\begin{defn}[Functionally-fitted EPTRKN]
An s-stage EPTRKN method is a functionally-fitted (or generalized 
collocation) explicit pseudo two-step RKN (FEPTRKN) method with respect to the basis functions 
$\{u_i(t)\}_{i=1}^s$ if the following relations are satisfied for all 
$i=1,...,s$:
\begin{equation}
\label{FEPTRKNdef}
\begin{split}
u_i(t+h)&=u_i(t) + hu_i'(t) + h^2\bm{b}^T(t,h)u''_i(\bm{e}t+\bm{c}h),\\
u'_i(t+h)&=u'_i(t) + h\bm{d}^T(t,h)u''_i(\bm{e}t+\bm{c}h),\\
u_i(\bm{e}t+\bm{c}h+\bm{e}h)&=\bm{e}u_i(t+h)+h\bm{c}u'_i(t+h) + h^2\bm{A}(t,h)u''_i(\bm{e}t+\bm{c}h).
\end{split}
\end{equation}
\end{defn}
Here, $\bm{b}^T(t,h)$ and $\bm{d}^T(t,h)$ denote the transposes of $\bm{b}(t,h)$ and $\bm{d}(t,h)$, respectively. 
This immediately yields linear systems to solve for $\bm{A}$, $\bm{b}$, and $\bm{d}$. These  coefficients generally
depend on both $t$ and $h$. The parameters $(c_i)_{i = 1}^s$ are distinct and will be chosen later. 
By construction, FEPTRKN methods are explicit.

\subsection{Collocation condition}

It is clear that not all bases of functions $\{u_i(t)\}_{i=1}^s$ satisfy 
\eqref{FEPTRKNdef}, i.e., 
there are some sets of functions $\{u_i(t)\}_{i=1}^s$ for which one cannot find 
$\bm{A}$, $\bm{b}$, and $\bm{d}$ 
so that condition \eqref{FEPTRKNdef} holds for these $\{u_i(t)\}_{i=1}^s$. The {\it collocation condition} defined below 
guarantees the existence of $\bm{A}$, $\bm{b}$ and $\bm{d}$ satisfying condition \eqref{FEPTRKNdef}, and, therefore, guarantees the existence of  
the corresponding FEPTRKN method.

\begin{defn}[Collocation condition]
A set of sufficiently smooth functions $\{u_1(t),u_2(t),...,$ $u_s(t)\}$ 
is said to satisfy the collocation condition if for any given $t\in [0,T)$, the matrix 
\begin{align*}
\bm{F}(t,h)&:=
\bigg{(}u''_{1}(\bm{e}t+\bm{c}h),
u''_{2}(\bm{e}t+\bm{c}h),...,
u''_{s}(\bm{e}t+\bm{c}h)\bigg{)}
\end{align*}
is nonsingular for almost every $h$ in the interval $[0,h_{\max}]$, $h_{\max} = const>0$.
\end{defn}

\begin{rem}\label{sysAb}{\rm
Equations in \eqref{FEPTRKNdef} can be written as
\begin{equation}
\label{sysABeq}
\begin{split}
\Big{(}u_{1}(t+h)-u_1(t)-hu'_1(t), ..., u_{s}(t+h)-u_s(t)-hu'_s(t)\Big{)} 
&= h^2\bm{b}^T(t, h)\bm{F}(t, h), \\
\Big{(}u'_{1}(t+h)-u'_1(t), ..., u'_{s}(t+h)-u'_s(t)\Big{)} 
&= h\bm{d}^T(t, h)\bm{F}(t, h), \\
\Big{(}\bm{v}_1(t,h), ...,
\bm{v}_s(t,h)\Big{)} 
&= h^2\bm{A}(t, h)\bm{F}(t, h),
\end{split}
\end{equation}
where
$$
\bm{v}_i(t,h):=u_{i}(\bm{e}t +\bm{e}h+\bm{c}h)-\bm{e}u_i(t+h)-h\bm{c}u'_i(t+h),\qquad i=1,...,s.
$$
Hence,  the existence of the coefficients $\bm{A}(t,h)$, $\bm{d}(t,h)$, $\bm{b}(t,h)$ at $t\in [t_0,t_0+T]$  and $h>0$ 
is guaranteed if $\bm{F}(t,h)$ is nonsingular. }
\end{rem}

\begin{rem}{\rm 
If $\{u_i(t)\}_{i=1}^s$ satisfies the collocation condition, then the linear function $\alpha + \beta t$ is not in the linear space $\Span\{u_1,u_2,...,u_s\}$ for any constants $\alpha$ and $\beta$. In other words, there do not exist constants $(\alpha_i)_{i=1}^s$ so that $\alpha_1u_1(t)+ \alpha_2u_2(t)+\cdots+\alpha_su_s(t) =\alpha + \beta t$ for some constants $\alpha$ and $\beta$. 
In particular, none of the functions $\{u_i(t)\}_{i=1}^s$ is a linear function. 
}
\end{rem}

\begin{rem}\rm 
When $\{u_1(t), u_2(t), ..., u_s(t)\}=\{t^2,t^3, ..., t^{s+1}\}$, 
we have 
$$
\bm{F}(t,h)=
\bigg{(}2\bm{e},
6(\bm{e}t+\bm{c}h),...,
(s+1)s(\bm{e}t+\bm{c}h)^{s-1}\bigg{)}.
$$
Subtracting between columns in the matrix $\bm{F}(t,h)$, we see that
$$
\det \bm{F}(t,h)=
\begin{vmatrix}
2  &6c_1h  & \cdots &(s+1)s(c_1h)^{s-1}\\
2  &6c_2h  & \cdots &(s+1)s(c_2h)^{s-1}\\
\vdots    & \vdots    & \ddots & \vdots\\
2  &6c_sh  & \cdots &(s+1)s(c_sh)^{s-1}
\end{vmatrix} 
= s!(s+1)! h^{\frac{s(s-1)}{2}}
\begin{vmatrix}
1  &c_1  & \cdots &c_1^{s-1}\\
1  &c_2  & \cdots &c_2^{s-1}\\
\vdots    & \vdots    & \ddots & \vdots\\
1  &c_s  & \cdots &c_s^{s-1}
\end{vmatrix}.
$$
Observe that $\det \bm{F}(t,h)$ is a multiple of the determinant of a Vandermonde matrix and, therefore, 
is nonzero if $h\not =0$ and $(c_i)_{i=1}^s$ are distinct.  
Hence, the collocation condition is always satisfied
when the fitting functions $\{u_i(t)\}_{i=1}^s$ are the monomials. 
\end{rem}

The practical implication of the proposed collocation condition is that 
for any given $t$ we can control the stepsize $h$ to get a nonsingular 
$\bm{F}(t,h)$ and then solve for $\bm{A},\bm{b}$, and $\bm{d}$.
For a given set of functions 
$\{u_1(t), ..., u_s(t)\}$ which satisfies the collocation condition and a given $t\in [t_0,t_0+T)$, there will be at most 
countably many values of $h$ such that $\bm{F}(t,h)$ is singular. Thus, we have the following result.

\begin{thm}
The coefficients of a FEPTRKN method with respect to a set 
of basis functions that satisfy the collocation condition are uniquely 
determined almost everywhere on the integration domain.
\end{thm}

\subsection{The collocation solution}

Let $(c_i)_{i=1}^s$ be distinct and $\{u_i\}_{i=1}^s$ be a basis satisfying the collocation condition. 
Define
$$\bm{H}:= \Span\{1,t,u_1,...,u_s\} :=
\bigg\{ \alpha_0 + a_0 t+
\sum_{i=1}^sa_iu_i(t)\bigg{|} \alpha_0\in \IR,\, a_i \in \IR, \, i=0,...,s\bigg\}.
$$
Given $y_n$, $y_n'$, and $f(t_n+c_ih,Y_{n,i})$, $i=1,...,s$, we call $u(t)$ the {\em collocation solution}
if $u(t)\in \bm{H}$ and the following equations hold
\begin{equation}
\label{interpol}
\begin{split}
u(t_n) &= y_n,\\
u'(t_n) &= y'_n,\\
u''(t_n+c_ih) &= f(t_n+c_ih,Y_{n,i}),\qquad i = 1, ..., s.
\end{split}
\end{equation}
If such a $u(t)$ exists, the numerical solution, derivative of the solution, and the stage values 
at the $(n+1)$-th step are defined by
\begin{equation}
\label{iteratepoly}
y_{n+1} := u(t_{n+1}),\quad y'_{n+1} := u'(t_{n+1}),\quad Y_{n+1,i} := u(t_{n+1}+c_ih),\qquad i=1,...,s, \quad t_{n+1}=t_n+h.
\end{equation}
Equations \eqref{interpol} and \eqref{iteratepoly} can be called a generalized collocation method for 
integrating equation \eqref{sys}. 
When $y_n$, $y'_n$, and $\bm{Y}_n= (Y_{n,1}, ..., Y_{n,s})^T$ are known, 
$u(t)$ can be constructed explicitly  
through a direct interpolation involving $y_n$, $y'_n$, and 
$f(\bm{e}t_n+\bm{c}h,\bm{Y}_n)$.
The existence of such collocation solution 
is made more precise as follows.

\begin{lem}\label{newinter}
Suppose that the $s+2$ values $z_0,z'_0,z_1,z_2 ..., z_s$ are given 
and the pair $(t_n,h)$ is such that $\bm{F}(t_n,h)$ is nonsingular, 
then there exists an interpolation function $\varphi \in\bm{H}$ such 
that 
\begin{equation}
\label{ae11}
\varphi(t_n) = z_0,\quad \varphi'(t_n) = z'_0,\quad \varphi''(t_n+c_ih)=z_i, \qquad i=1,...,s.
\end{equation}

\end{lem}
\begin{proof}
Since $\varphi \in \bm{H}$, it can
be represented in the form
$$
\varphi(t)=\alpha_0+ a_0t+\sum_{i=1}^{s}a_iu_i(t).
$$
Let $t_{n,i}:=t_n+c_ih$, $i=1,...,s$. 
Equation \eqref{ae11}
can be written as
\begin{equation}
\label{mq1}
\begin{pmatrix}
1&t_n&u_{1}(t_n)  & u_{2}(t_n)  & \cdots & u_{s}(t_n)\\
0&1&u'_{1}(t_n)  & u'_{2}(t_n)  & \cdots & u'_{s}(t_n)\\
0&0&u''_{1}(t_{n,1})  & u''_{2}(t_{n,1})  & \cdots & u''_{s}(t_{n,1})\\
\vdots&\vdots&\vdots    & \vdots    & \ddots & \vdots\\
0&0&u''_{1}(t_{n,s})  & u''_{2}(t_{n,s})  & \cdots & u''_{s}(t_{n,s})
\end{pmatrix}
\begin{pmatrix}
\alpha_0\\
a_0\\
a_1\\
\vdots\\
a_s
\end{pmatrix}
=
\begin{pmatrix}
z_0\\
z'_0\\
z_1\\
\vdots\\
z_s
\end{pmatrix}.
\end{equation}
Since the matrix in the left-hand side of \eqref{mq1} has the same
determinant as $\bm{F}(t_n,h)$, the constants $(a_i)_{i=0}^s$ and $\alpha_0$
are uniquely determined from linear system \eqref{mq1} under the assumption that $\bm{F}(t_n,h)$
is nonsingular. Thus, the function $\varphi(t)\in \bm{H}$ satisfying \eqref{ae11} is determined uniquely. 
\end{proof}

\begin{thm}
\label{collofuns}
The generalized collocation method \eqref{interpol}--\eqref{iteratepoly} 
is equivalent to the $s$-stage FEPTRKN method with coefficients 
$(\bm{c}$, $\bm{A}(t_n,h)$, $\bm{b}(t_n,h)$, $\bm{d}(t_n,h))$.
\end{thm}

\begin{proof}
It follows from Lemma~\ref{newinter} that there exists a unique interpolation
function $u(t)\in \bm{H}$ such that 
\begin{equation}
\label{mr1}
u(t_{n})=y_{n},\qquad u'(t_{n})=y'_{n},\qquad 
 u''(t_{n}+c_ih)=f(t_{n}+ c_ih,Y_{n,i}),\qquad i=1,...,s. 
\end{equation}
Thus, $u(t)$ satisfies \eqref{interpol} and is the collocation solution being searched for. To complete the proof,
we have to prove that if we use $u(t)$ to generate the quantities in \eqref{iteratepoly},
then the following relations that define the FEPTRKN method hold true:
\begin{equation}
\label{eq10}
\begin{split}
y_{n+1} &= y_{n} +hy'_n+ h^2\bm{b}^T(t_{n},h)f(\bm{e}t_{n}+\bm{c}h,\bm{Y}_{n}),\\
y'_{n+1} &= y'_n+ h\bm{d}^T(t_{n},h)f(\bm{e}t_{n}+\bm{c}h,\bm{Y}_{n}),\\
\bm{Y}_{n+1} &= \bm{e}y_{n+1} +h\bm{c}y'_{n+1} + h^2\bm{A}(t_{n},h)f(\bm{e}t_{n}+\bm{c}h,\bm{Y}_{n}).
\end{split}
\end{equation}
Since $u(t) \in \bm{H}$, it has the following representation
\begin{equation}
\label{phicomb}
u(t)=\alpha_0 + a_0t + \sum_{i=1}^sa_iu_i(t).
\end{equation}
From the definition of a FEPTRKN method in \eqref{FEPTRKNdef},
the coefficients $(\bm{c},  \bm{A}, \bm{b},\bm{d})$ 
satisfy
\begin{subequations}
\begin{align}
\label{eq12a}
u_i(t_{n+1}) &= u_i(t_{n}) + h u'_i(t_{n})+
h^2\bm{b}^T(t_{n},h)u_i''(\bm{e}t_{n}+\bm{c}h),\\
u'_i(t_{n+1}) &= u'_i(t_{n}) +
h\bm{d}^T(t_{n},h)u_i''(\bm{e}t_{n}+\bm{c}h),\\
u_i(\bm{e}t_n+\bm{e}h+\bm{c}h) &= \bm{e}u_i(t_n+h) + \bm{c}hu'_i(t_n+h)+
h^2\bm{A}(t_n,h)u''_i(\bm{e}t_n+\bm{c}h),
\end{align}
\end{subequations}
for all $i = 1, ..., s$.
From equation \eqref{eq12a} and the fact that $u(t)$ is a linear 
combination of functions $u_i(t)$, $1$, and $t$ (cf. \eqref{phicomb}), we obtain
\begin{align}
\label{cuoi2}
u(t_{n+1})=u(t_{n})+hu'(t_{n})+ h^2\bm{b}^T(t_{n},h)u''(\bm{e}t_{n}+\bm{c}h).
\end{align}
Equation \eqref{iteratepoly}, equation \eqref{mr1}, and equation
\eqref{cuoi2} imply that
\begin{align*}
\label{equivalent1}
y_{n+1} = u(t_{n+1}) = y_n + hy'_n +
h^2\bm{b}^T(t_{n},h)f(\bm{e}t_{n}+\bm{c}h,\bm{Y}_{n}).
\end{align*}
Thus, the first equation in \eqref{eq10} holds. The other equations in \eqref{eq10} are obtained similarly. Theorem \ref{collofuns} is proved. 
\end{proof}

\section{Accuracy and stability properties}
\label{sec:accuracy}

\subsection{Accuracy order}

The stage order and step order of a FEPTRKN method are defined
as follows (cf. \cite{cong99}). 

\begin{defn}
\label{def4.1}
A FEPTRKN method is said to have step order $p=\min\{p_1,p_2\}$ 
and stage order $r=\min\{p_1,p_2,p_3\}$ if
\begin{align*}
y(t_{n+1})-y_{n+1}&=O(h^{p_1+1}),\\
y'(t_{n+1})-y'_{n+1}&=O(h^{p_2+1}),\\
y(\bm{e}t_{n+1}+\bm{c}h)-\bm{Y}_{n+1}&=O(h^{p_3+1}),
\end{align*}
given that
$y_{n} = y(t_{n})$, $y'_{n} = y'(t_{n})$, and $y(\bm{e}t_{n}+\bm{c}h)-\bm{Y}_{n}=O(h^{p_3+1})$.
\end{defn}

\begin{rem}\label{smooth1}\rm 
We recall this handy result from \cite[Remark 3.1]{nguyen2}.
If a function $f(t)\in C^{m+n}([t_0,\,t_0+T])$ satisfies  
$f(t_i)=0,i=1,...,n$, then there exists $g(t)\in C^{m}([t_0,\,t_0+T])$ such that 
$f(t)=g(t)\prod_{i=1}^n(t-t_i)$. 
\end{rem}

\begin{thm}\label{genorder}
An s-stage FEPTRKN method has stage order $r=s$ and 
step order at least $p=s$.
\end{thm}

\begin{proof}
Without loss of generality we assume that $t_{n}=0$. 
Let $u(t)\in \bm{H}$ be the collocation solution satisfying
\begin{equation}
\label{ae9}
u(t_{n})=y_{n},\quad u'(t_{n})=y'_{n},\quad
u''(\bm{c}h)=f(\bm{c}h,\bm{Y}_{n}).
\end{equation}
This collocation solution exists by Lemma~\ref{newinter}  
and we have
$$
y_{n+1} = u(t_{n+1}),\quad y'_{n+1}=u'(t_{n+1}),\quad \bm{Y}_{n+1}=u(\bm{e}h+\bm{c}h).
$$
From the equation $y''(t) = f(t,y(t))$, $\forall t\in [t_0,t_0 + T]$, the third equation in \eqref{ae9}, the local assumption 
$\bm{Y}_{n}-y(\bm{c}h) = O(h^{s+1})$,
and Taylor expansions of $f(t,y)$ with respect to $y$, one gets
\begin{equation}
\label{ae10}
u''(\bm{c}h)-y''(\bm{c}h)=
f(\bm{c}h,\bm{Y}_{n})-f(\bm{c}h,y(\bm{c}h))
= O(h^{s+1}).
\end{equation}
Let 
$$
R(t):=u(t)-y(t),\qquad 0\le t\le t_0 + T. 
$$ 
From equations \eqref{ae9} and \eqref{ae10} one gets
$$
R(0)=0,\qquad R'(0)=0,\qquad R''(c_ih)=u''(c_ih) - y''(c_ih) =O(h^{s+1}),\qquad i=1,...,s.
$$
Since $R''(t)$ is sufficiently smooth and $R''(c_ih)=O(h^{s+1})$, $i=1,...,s$, there exists a sufficiently smooth function $w(t)$
such that 
$$
R''(t)=w(t)+O(h^{s+1}),\qquad  w(c_ih)=0,\qquad i=1,..,s.
$$ 
Since $w(c_ih)=0$, $i=1,...,s$, it follows from Remark~\ref{smooth1} that there exists a sufficiently smooth function $g(t)$ such that
$$
w(t) = g(t)\prod_{i=1}^s(t-c_ih).
$$
Thus,
\begin{equation}
\label{eq14}
R''(t) = g(t)\prod_{i=1}^s(t-c_ih) + O(h^{s+1}).
\end{equation}
This and the relation $R'(0)=0$ imply
$$
R'(x) = \int_0^x R''(t)\, dt = \int_0^x \bigg(g(t)\prod_{i=1}^s(t-c_ih) + O(h^{s+1})\bigg)\, dt.
$$
Letting $x = h$ in the above equality and using the substitution $t=\xi h$, we get
\begin{equation}
\label{eq16}
R'(h) = h^{s+1}\int_0^1 g(\xi h)\prod_{i=1}^s(\xi-c_i)\, d\xi + 
\int_0^{h} O(h^{s+1})\, dt. 
\end{equation}
This and the relation $\int_0^{h} O(h^{s+1})\, dt = O(h^{s+2})$ imply
\begin{equation}
\label{ae6}
y'_{n+1}-y'(t_{n+1})
=u'(t_{n+1})-y'(t_{n+1})=R'(h) = O(h^{s+1}).
\end{equation}

From equation \eqref{eq14} and the equalities $R(0)=0$ and $R'(0)=0$, one gets
\begin{equation*}
R(x)=\int_0^x\int_0^t R''(\xi )\,d\xi dt =\int_0^x\int_0^t\bigg{(}g(\xi)\prod_{i=1}^s(\xi -c_ih)+ O(h^{s+1})\bigg{)}\, d\xi dt .
\end{equation*}
Substituting $x:=\omega h$ in the above equality and using substitutions, we obtain
\begin{equation}
\label{intererror}
\begin{split}
u(\omega h)-y(\omega h) = R(\omega h) &= h^{s+2}\int_0^\omega\int_0^t g(\xi h)\prod_{i=1}^s(\xi -c_i)\,d\xi dt +\int_0^{\omega h}\int_0^t O(h^{s+1})\, d\xi dt.
\end{split}
\end{equation}
This and the relation $\int_0^{\omega h}\int_0^t O(h^{s+1})\, d\xi dt= O(h^{s+3})$, $\omega=const$, imply
\begin{align}
\label{eq14b}
\bm{Y}_{n+1}-y(\bm{e}h+\bm{c}h)
&=u(\bm{e}h+\bm{c}h)-y(\bm{e}h+\bm{c}h)=O(h^{s+2}),\\
\label{eq14c}
y_{n+1}-y(h)
&=u(h)-y(h)=O(h^{s+2}).
\end{align}
Equations \eqref{ae6}, \eqref{eq14b}, \eqref{eq14c}, and Definition \ref{def4.1} imply that 
the method has step order $p=s$ and stage order $r=s$. 

Theorem \ref{genorder} is proved. 
\end{proof}

\begin{remark}{\rm 
\label{remarka4}
From \eqref{eq14b} we have
\begin{equation}
\bm{Y}_n - y(\bm{e}t_n +\bm{c}h) = O(h^{s+2}),\qquad n=1,2,...\,.
\end{equation}
Thus, if $\bm{Y}_0$ is computed with a local accuracy order of $s+2$ at the initial step, then we have
\begin{equation}
\label{aex1}
\bm{Y}_n - y(\bm{e}t_n +\bm{c}h) = O(h^{s+2}),\qquad n=0,1,...\,.
\end{equation}
}
\end{remark}

\subsection{Superconvergence}

\begin{thm}\label{super1}
If an $s$-stage FEPTRKN 
method has collocation parameters $(c_i)_{i = 1}^s$ that satisfy
\begin{equation}
\label{eq16a}
\int_{0}^1\prod_{i=1}^s(\xi-c_i)\,d\xi  = 0,
\end{equation}
then the method has step order $p=s+1$ and stage order $r=s+1$. 
\end{thm}

\begin{proof}From equation \eqref{eq16} and the relation $\int_0^h O(h^{s+1})\, dt = O(h^{s+2})$, we have
$$
y'_{n+1}-y'(t_{n+1}) = u'(t_{n+1})-y'(t_{n+1})=R'(h)=h^{s+1}\int_0^1g(\xi h)\prod_{i=1}^s(\xi -c_i)\,d\xi + O(h^{s+2}).
$$
Use the Taylor expansion $g(\xi h)=g(0)+O(\xi h)$ and orthogonality condition \eqref{eq16a} to get
\begin{equation}
\label{eq16c}
y'_{n+1} -y'(t_{n+1})=h^{s+1}g(0)\int_0^1\prod_{i=1}^s(\xi -c_i)\,d\xi + O(h^{s+2}) = O(h^{s+2}).
\end{equation}
The conclusion of the theorem follows from equations \eqref{eq14b}, \eqref{eq14c}, \eqref{eq16c}, and Definition \ref{def4.1}.
\end{proof}

\begin{thm}\label{super2}
If an $s$-stage FEPTRKN 
method has collocation parameters $(c_i)_{i = 1}^s$ that satisfy
\begin{equation}
\label{eq17}
\int_{0}^1 \xi^k\prod_{i=1}^s(\xi-c_i)\, d\xi = 0,\qquad k=0,1,
\end{equation}
then the method has step order $p=s+2$. 
\end{thm}

\begin{proof}
From the relation $\bm{Y}_n - y(\bm{c}h)= O(h^{s+2})$ (see Remark \ref{remarka4}) and Taylor expansions of $f(t,y)$, one gets
\begin{equation}
\label{mq2}
f(\bm{c}h,\bm{Y}_{n})-f(\bm{c}h,y(\bm{c}h))= O(h^{s+2}).
\end{equation}
Let $R(t)$ be defined as in Theorem \ref{genorder}, i.e., $R(t):=u(t) - y(t)$. Then, we have $R(0)=0$ and $R'(0)=0$. It follows from \eqref{mq2} that 
$$
R''(c_i h) = u''(c_i h) - y''(c_i h) = f(c_ih, Y_{n,i})-f(h,y(c_i h))= O(h^{s+2}), \qquad i=1,...,s.
$$  
Thus, there exists a sufficiently smooth function $w(t)$ such that $w(c_ih)=0$, $i=1,...,s$, and
\begin{equation}
R''(t)=w(t) +O(h^{s+2}),\qquad 0\le t \le t_0 +T.
\end{equation}
 Again, there exists 
a sufficiently smooth function $g(t)$ such that $w(t)=g(t)\prod_{i=1}^s(t-c_ih)$. 
Therefore,
\begin{equation}
\label{aex2}
R''(t) = g(t)\prod_{i=1}^s(t-c_ih) + O(h^{s+2}),\qquad 0\le t \le t_0 +T.
\end{equation}
Equation \eqref{aex2} and the relation $R'(0)=0$ imply
\begin{equation}
\label{ae2}
R'(x)=\int_0^xR''(\xi)\,d\xi =\int_0^x\bigg{(}
g(\xi)\prod_{i=1}^s(\xi -c_ih)+O(h^{s+2})\bigg{)}\,d\xi.
\end{equation}
Assigning $x:=h$ in the above equality, we obtain
\begin{equation}
\label{intererror2}
u'(h) - y'(h)=R'(h) = h^{s+1}\int_{0}^{1}
g(\xi h)\prod_{i=1}^s(\xi -c_i)\,d\xi + \int_0^{h}O(h^{s+2})\, d\xi.
\end{equation}
Using the Taylor expansion $g(\xi h)=g(0)+\xi hg'(0)+O(h^2)$ and the relation $\int_0^{h}O(h^{s+2})\, d\xi = O(h^{s+3})$, we have
\begin{align*}
u'(h)-y'(h)&=h^{s+1}g(0)\int_0^{1}
\prod_{i=1}^s(\xi -c_i)\,d\xi +h^{s+2}g'(0)
\int_0^{1}\xi \prod_{i=1}^s(\xi -c_i)\,d\xi  + O(h^{s+3}).
\end{align*}
This and orthogonality condition \eqref{eq17} imply
\begin{equation}
\label{ae4}
y'_{n+1} - y'(t_{n+1}) = u'(h)-y'(h) = O(h^{s+3}). 
\end{equation}

Similarly, from equation \eqref{aex2} and the relations $R(0) = 0$ and $R'(0) = 0$, one gets
\begin{equation}
\label{eq19}
R(x) = \int_0^x\int_0^t R''(\xi) \,d\xi dt= \int_0^x\int_0^t\bigg{(}
g(\xi)\prod_{i=1}^s(\xi -c_ih)+O(h^{s+2})\bigg{)}\,d\xi dt.
\end{equation}
From equation \eqref{eq19}, the relation $\int_0^h\int_0^t O(h^{s+2})\, d\xi dt = O(h^{s+4})$, Fubini's Theorem, the Maclaurin expansion $g(\xi h) = g(0) + O(h)$, and orthogonality condition \eqref{eq17}, one gets
\begin{equation}
\label{ae7x}
\begin{split}
y_{n+1} - y (t_{n+1}) =R(h) &= h^{s+2}\int_0^1\int_0^t g(\xi h)\prod_{i=1}^s(\xi -c_i)\,d\xi dt +O(h^{s+4})\\
&= h^{s+2}\int_0^1\int_\xi^1 g(\xi h)\prod_{i=1}^s(\xi -c_i)\,dt d\xi +O(h^{s+4})\\
&= h^{s+2}\int_0^1 g(\xi h)(1-\xi)\prod_{i=1}^s(\xi -c_i)\, d\xi +O(h^{s+4}) \\
&= h^{s+2}g(0)\bigg(\int_0^1 \prod_{i=1}^s(\xi -c_i)\, d\xi - \int_0^1 \xi\prod_{i=1}^s(\xi -c_i)\, d\xi\bigg) +O(h^{s+3})  \\
&= O(h^{s+3}).
\end{split}
\end{equation}
From \eqref{ae4}, \eqref{ae7x}, and Definition \ref{def4.1}, one concludes that the method has step order $p=s+2$. 

Theorem \ref{super2} is proved. 
\end{proof}

\begin{rem}
From \eqref{intererror} with $w=1+c_i$, $i=1,...,s$, and the relation $\int_0^{\omega h}\int_0^t O(h^{s+1})\, d\xi dt = O(h^{s+3})$, we have
\begin{equation}
\begin{split}
\bm{Y}_{n+1}-y(\bm{e}t_{n+1}+\bm{c}h) &= u(\bm{e}h +\bm{c}h) -  y(\bm{e}h +\bm{c}h) \\
&=
h^{s+2}C_{n+1}\int_{\bm{0}}^{\bm{e}+\bm{c}}\int_0^t
\prod_{i=1}^s(\xi -c_i)\,d\xi dt + O(h^{s+3}),\qquad n=0,1,...\,.
\end{split}
\end{equation}
If $\bm{Y}_0$ is computed at a local accuracy order of $s+2$ at the initial step of computation, i.e., $\bm{Y}_0 - y(\bm{e}t_0 + \bm{c}h) = O(h^{s+3})$, then we will have
\begin{equation}
\label{stagerrorx}
\bm{Y}_{n}-y(\bm{e}t_{n} + \bm{c}h)=
h^{s+2}C_{n}\int_{\bm{0}}^{\bm{e}+\bm{c}}\int_0^t
\prod_{i=1}^s(\xi -c_i)\,d\xi dt + O(h^{s+3}),\qquad n=0,1,...\, .
\end{equation}
\end{rem}

\begin{thm}
\label{super3}
If an $s$-stage FEPTRKN 
method has collocation parameters $(c_i)_{i = 1}^s$ that satisfy
\begin{equation}
\label{eq20}
\int_{0}^1\int_{0}^{1+x}\int_0^t\prod_{i=1}^s(\xi-c_i)\,d\xi dt dx = 0,\qquad
\int_{0}^1 \xi^k\prod_{i=1}^s(\xi-c_i)\,d\xi = 0,\quad k=0,1,2,
\end{equation}
then the method has step order $p=s+3$. 
\end{thm}

\begin{proof}
Again, without loss of generality, we assume that $t_n = 0$. 
Let $R(t)$ be defined as in the proofs of Theorems \ref{genorder} and \ref{super2}, i.e., $R(t)=u(t) -y(t)$. 
Using Taylor expansions of $f(t,y)$ and \eqref{stagerrorx}, one gets
\begin{equation*}
\begin{split}
R''(\bm{c}h) = u''(\bm{c}h) - y''(\bm{c}h)= f(\bm{c}h,\bm{Y}_{n})-f(\bm{c}h,y(\bm{c}h))
=
h^{s+2} C
\int_{\bm{0}}^{\bm{e}+\bm{c}}\int_0^t\prod_{i=1}^s(\xi -c_i)\,d\xi dt + O(h^{s+3}),
\end{split}
\end{equation*}
where $C=C_n\frac{\partial f}{\partial y}(0,y(0))$. 
Thus, we have
$$
R(0) = 0,\qquad R'(0) = 0,\qquad R''(c_ih) = C
\int_{0}^{h+c_ih}\int_0^t\prod_{i=1}^s(\xi -c_ih)\,d\xi dt + O(h^{s+3}),\qquad i=1,...,s.
$$
Again, since $R''(t)$ is sufficiently smooth, there exists a smooth function $w(t)$ such that
\begin{equation}
\label{a11}
R''(x) = w(x) + C\int_{0}^{h+x}\int_0^t\prod_{i=1}^s(\xi -c_ih)\,d\xi dt + O(h^{s+3}),\qquad w(c_ih) = 0,\qquad i=1,...,s. 
\end{equation}
Since $w(c_ih)=0$, $i=1,...,s$, there exists a sufficiently smooth function $g(x)$ such that (cf. Remark~\ref{smooth1})
$$
w(x) = g(x)\prod_{i=1}^s(x-c_ih).
$$
This and \eqref{a11} imply
\begin{equation}
\label{eq22}
\begin{split}
R''(x)&= g(x)\prod_{i=1}^s(x -c_ih) + C\int_{0}^{h+x}\int_0^t\prod_{i=1}^s(\xi -c_ih) \,d\xi dt + O(h^{s+3}).
\end{split}
\end{equation}
Integrate equation \eqref{eq22} and use the relation $R'(0)=0$ to get
\begin{equation*}
\begin{split}
R'(x)&=\int_0^xR''(\tau)d\tau \\
&=\int_0^x\bigg{(}
g(\tau)\prod_{i=1}^s(\tau -c_ih) + C\int_{0}^{h+\tau}\int_0^t\prod_{i=1}^s(\xi -c_ih)\,d\xi dt  + O(h^{s+3}) \bigg{)}d\tau\\
&=\int_0^x
g(\tau)\prod_{i=1}^s(\tau -c_ih)\, d\tau + C\int_0^x\int_{0}^{h+\tau}\int_0^t\prod_{i=1}^s(\xi -c_ih)\,d\xi dt  d\tau + \int_0^xO(h^{s+3})\, d\tau.
\end{split}
\end{equation*}
Assigning $x:= h$ in the above equality, we obtain
\begin{equation}
\label{eq23}
\begin{split}
u'(h)-y'(h) = 
R'(h) &= h^{s+1}\int_0^1 g(h\xi)
\prod_{i=1}^s(\xi -c_i)\, d\xi \\
&\, + Ch^{s+3}\int_0^1\int_{0}^{1+\tau}\int_0^t\prod_{i=1}^s(\xi -c_i)\,d\xi dt  d\tau + \int_0^{h}O(h^{s+3})\,d\tau.
\end{split}
\end{equation}
Using the Taylor expansion $g(\xi h)=\sum_{k=0}^2 \frac{g^{(k)}(0)}{k!}(\xi h)^k +O(h^3)$, we get
\begin{align}
\label{mw1}
h^{s+1}\int_0^1g(h\xi)
\prod_{i=1}^s(\xi -c_i)\, d\xi = \sum_{k=0}^2 h^{s+1+k}\frac{g^{(k)}(0)}{k!}
\int_0^{1}\xi^k \prod_{i=1}^s(\xi -c_i)d\xi  + O(h^{s+4}).
\end{align}
Equations \eqref{eq23}--\eqref{mw1} and the relation $\int_0^{h}O(h^{s+3})\,d\tau=O(h^{s+4})$ imply
\begin{equation*}
\begin{split}
u'(h)-y'(h) 
 &= \sum_{k=0}^2 h^{s+1+k}\frac{g^{(k)}(0)}{k!}
\int_0^{1}\xi^k \prod_{i=1}^s(\xi -c_i)d\xi \\
&\, + Ch^{s+3}\int_0^1\int_{0}^{1+\tau}\int_0^t\prod_{i=1}^s(\xi -c_i)\,d\xi dt  d\tau + O(h^{s+4}).
\end{split}
\end{equation*}
This and orthogonality conditions in \eqref{eq20} imply
\begin{equation}
\label{ae7}
y'_{n+1} - y(t_{n+1}) = 
u'(h)-y'(h) = O(h^{s+4}). 
\end{equation}

From equation \eqref{eq22} and the relations $R(0)=0$ and $R'(0)=0$, one gets
\begin{equation}
\begin{split}
R(x)  &=\int_0^x\int_0^\zeta R''(\tau)d\tau d\zeta \\
&=\int_0^x\int_0^\zeta \bigg{(}
g(\tau)\prod_{i=1}^s(\tau -c_ih) + C\int_{0}^{h+\tau}\int_0^t\prod_{i=1}^s(\xi -c_ih)\,d\xi dt  + O(h^{s+3}) \bigg{)}d\tau d\zeta.
\end{split}
\end{equation}
This and the relation $\int_0^h\int_0^\zeta O(h^{s+3})\, d\tau d\zeta = O(h^{s+5})$ imply
\begin{equation}
\label{eq24}
\begin{split}
R(h) &=\int_0^h \int_0^\zeta
g(\tau)\prod_{i=1}^s(\tau -c_ih)\, d\tau d\zeta + C\int_0^h\int_0^\zeta\int_{0}^{h+\tau}\int_0^t\prod_{i=1}^s(\xi -c_ih)\,d\xi dt  d\tau d\zeta + O(h^{s+5}).
\end{split}
\end{equation}
By substitutions, one gets
\begin{equation}
\label{eq25}
\int_0^h\int_0^\zeta\int_{0}^{h+\tau}\int_0^t\prod_{i=1}^s(\xi -c_ih)\,d\xi dt  d\tau d\zeta
= h^{s+4} \int_0^1\int_0^\zeta\int_{0}^{1+\tau}\int_0^t\prod_{i=1}^s(\xi -c_i)\,d\xi dt  d\tau d\zeta = O(h^{s+4}). 
\end{equation}
From Fubini's Theorem one obtains
\begin{equation}
\label{eq26}
\begin{split}
\int_0^h \int_0^\zeta g(\tau)\prod_{i=1}^s(\tau -c_ih)\, d\tau d\zeta 
&= \int_0^h \int_\tau^h g(\tau)\prod_{i=1}^s(\tau -c_ih)\, d\zeta d\tau \\
&= \int_0^h g(\tau)(h-\tau)\prod_{i=1}^s(\tau -c_ih)\,d\tau. 
\end{split}
\end{equation}
Using the substitution $\tau=h\xi$ and the Maclaurin expansion $g(t) = \beta_0 +\beta_1t + O(t^2)$, one gets
\begin{equation}
\label{eq27}
\begin{split}
\int_0^h g(\tau)(h-\tau)\prod_{i=1}^s(\tau -c_ih)\,d\tau = & h^{s+2}\int_0^1 g(\xi h)(1-\xi)\prod_{i=1}^s(\xi -c_i)\,d\xi\\
=& \sum_{k=0}^1 \beta_k h^{s+2+k}
\int_0^1 \xi^k\prod_{i=1}^s(\xi -c_i)\,d\xi \\
& - \sum_{k=0}^1 \beta_k h^{s+2+k}
\int_0^1 \xi^{k+1}\prod_{i=1}^s(\xi -c_i)\,d\xi + O(h^{s+4}).
\end{split}
\end{equation}
From equations \eqref{eq26}--\eqref{eq27} and orthogonality conditions in \eqref{eq20}, one obtains
$$
\int_0^h \int_0^\zeta g(\tau)\prod_{i=1}^s(\tau -c_ih)\, d\tau d\zeta  = O(h^{s+4}).
$$
This and equations \eqref{eq24}--\eqref{eq25} imply
\begin{equation}
\label{ae8}
y_{n+1} - y(t_{n+1}) = u(h) - y(h) = R(h) = O(h^{s+4}).
\end{equation}
It follows from \eqref{ae7}, \eqref{ae8}, and Definition \ref{def4.1} that the stage order of the method is $p = s+3$. 

Theorem \eqref{super3} is proved. 
\end{proof}

\begin{rem}\rm 
From Fubini's Theorem, we have
\begin{align*}
\int_{0}^1\int_{0}^{1+x}\int_0^t\prod_{i=1}^s(\xi-c_i)\, d\xi dt dx &= 
\int_{0}^1\int_{0}^{1+x}\int_\xi^{1+x}\prod_{i=1}^s(\xi-c_i)\, dt d\xi  dx \\
&= 
\int_{0}^1\int_{0}^{1+x}(1+x - \xi)\prod_{i=1}^s(\xi-c_i)\, d\xi  dx.
\end{align*}
If \eqref{eq20} holds, then by Fubini's Theorem one gets
\begin{align*}
\int_{0}^1\int_{0}^{1}(1+x - \xi)\prod_{i=1}^s(\xi-c_i)\, d\xi  dx &= \int_{0}^1\int_{0}^{1}(1+x - \xi)\prod_{i=1}^s(\xi-c_i)\,   dx d\xi\\
&= \int_{0}^1\bigg(\frac{3}{2}- \xi\bigg)\prod_{i=1}^s(\xi-c_i)\,  d\xi\\
&= \frac{3}{2}\int_{0}^1\prod_{i=1}^s(\xi-c_i)\,  d\xi - \int_{0}^1\xi\prod_{i=1}^s(\xi-c_i)\,  d\xi = 0.
\end{align*}
Moreover, using Fubini's Theorem, one gets
\begin{align*}
\int_{0}^1\int_{1}^{1+x}(1+x - \xi)\prod_{i=1}^s(\xi-c_i)\, d\xi  dx &= 
\int_{1}^2\int_{\xi-1}^{1}(1+x - \xi)\prod_{i=1}^s(\xi-c_i)\, dx d\xi\\
&= 
\int_{1}^2\frac{(\xi - 2)^2}{2}\prod_{i=1}^s(\xi-c_i)\, d\xi.
\end{align*}
Therefore, condition \eqref{eq20} is equivalent to
\begin{equation}
\label{v48}
\int_{1}^2(\xi- 2)^2\prod_{i=1}^s(\xi-c_i)\, d\xi = 0,\qquad
\int_{0}^1 \xi^k\prod_{i=1}^s(\xi-c_i)\,d\xi = 0,\qquad k=0,1,2.
\end{equation}
In practice, we use \eqref{v48} instead of \eqref{eq20} to find collocation parameters $(c_i)_{i=1}^s$. 
\end{rem}

\subsection{Stability}
\label{subsec:stability}

Applying an FEPTRKN method given by $(\bm{c},\bm{A},\bm{b},\bm{d})$ to the Dahlquist test equation
$$
y'' = \lambda y,\qquad y(0) = y_0,\qquad y'(0) = y'_0,
$$
we get the iterations
\begin{equation}
\label{v49}
\begin{split}
\bm{Y}_{n+1} &= \bm{e}y_{n+1} + \bm{c}hy'_{n+1}+ \lambda h^2\bm{A}(t_n,h)\bm{Y}_n,\\
y_{n+1} &= y_n + hy'_n+ \lambda h^2\bm{b}^T(t_n,h)\bm{Y}_n,\\
y'_{n+1} &= y'_n + \lambda h\bm{d}^T(t_n,h)\bm{Y}_n.
\end{split}
\end{equation}
This implies
\begin{equation}
\label{v49.1}
\begin{pmatrix}
\bm{Y}_{n+1}\\
y_{n+1}\\
hy'_{n+1}
\end{pmatrix}
= \bm{M}_{(t_n,h)}(z)
\begin{pmatrix}
\bm{Y}_{n}\\
y_{n}\\
hy'_{n}
\end{pmatrix},\qquad z = \lambda h^2,
\end{equation}
where
$$
\bm{M}_{(t_n, h)}(z)=
\begin{pmatrix}
z\big[\bm{A}(t_n,h) + \bm{e}\bm{b}^T(t_n,h) + \bm{c} \bm{d}^T(t_n,h)\big]&\bm{e} & \bm{e}+\bm{c}\\
z\bm{b}^T(t_n,h) &1 & 1\\
z\bm{d}^T(t_n,h)  &0 & 1\\
\end{pmatrix}.
$$

Observe that the formula of our amplification matrix 
$\bm{M}_{(t_n,h)}(z)$ differs from that in~\cite{cong99,cong01}, where 
they instead use the formulation
\begin{equation}
\label{v50}
\begin{pmatrix}
\bm{Y}_{n}\\
y_{n+1}\\
hy'_{n+1}
\end{pmatrix}
= \tilde{\bm{M}}(z) 
\begin{pmatrix}
\bm{Y}_{n-1}\\
y_{n}\\
hy'_{n}
\end{pmatrix},\qquad 
\tilde{\bm{M}}(z) = 
\begin{pmatrix}
z\bm{A}&\bm{e} & \bm{c} \\
z^2\bm{b}^T\bm{A} &1+z\bm{b}^T\bm{e} & 1+z\bm{b}^T\bm{c}\\
z^2\bm{d}^T\bm{A}  & z\bm{d}^T\bm{e} & 1+z\bm{d}^T\bm{c}\\
\end{pmatrix},\qquad z=\lambda h^2.
\end{equation}
Note that in~\cite{cong99,cong01} the coefficients $\bm{A}(t,h)$, $\bm{b}(t,h)$ and $\bm{d}(t,h)$ are independent of $t$ and $h$. 
From \eqref{v49}, we have
$$
\begin{pmatrix}
\bm{Y}_{n+1}\\
y_{n+1}\\
hy'_{n+1}
\end{pmatrix}
=
\begin{pmatrix}
z\bm{A} & \bm{e}& \bm{c} &\\ 
0 & 1 & 0\\
0 & 0 & 1
\end{pmatrix}
\begin{pmatrix}
\bm{Y}_{n}\\
y_{n+1}\\
hy'_{n+1}
\end{pmatrix}.
$$
This and equation \eqref{v50} imply 
\begin{equation}
\label{v51}
\begin{pmatrix}
\bm{Y}_{n+1}\\
y_{n+1}\\
hy'_{n+1}
\end{pmatrix}
=
\begin{pmatrix}
z\bm{A} & \bm{e}& \bm{c} &\\ 
0 & 1 & 0\\
0 & 0 & 1
\end{pmatrix}
\tilde{\bm{M}}(z)
\begin{pmatrix}
z\bm{A} & \bm{e}& \bm{c} &\\ 
0 & 1 & 0\\
0 & 0 & 1
\end{pmatrix}^{-1}
\begin{pmatrix}
\bm{Y}_{n}\\
y_{n}\\
hy'_{n}
\end{pmatrix}.
\end{equation}
From \eqref{v49.1} and \eqref{v51}, we have the similarity transformation
$$
\bm{M}(z) =
\begin{pmatrix}
z\bm{A} & \bm{e}& \bm{c} &\\ 
0 & 1 & 0\\
0 & 0 & 1
\end{pmatrix}
\tilde{\bm{M}}(z)
\begin{pmatrix}
z\bm{A} & \bm{e}& \bm{c} &\\ 
0 & 1 & 0\\
0 & 0 & 1
\end{pmatrix}^{-1}.
$$
This means that although the two formulations lead to different 
amplification matrices, the spectral radii of these matrices are the same. Note that for 
conventional EPTRKN methods, we have $\bm{d}^T\bm{e}=1$ and $\bm{b}^T\bm{e}=\frac{1}{2}$ and, therefore, the matrix $\tilde{\bm{M}}(z)$ can be simplified further.  However, for  
general basis functions, the coefficients $\bm{A}$, $\bm{b}$, and $\bm{d}$ of a FEPTRKN method depend on both $t$ and $h$ and we do 
not have $\bm{d}^T\bm{e}=1$ and $\bm{b}^T\bm{e}=\frac{1}{2}$.

The coefficients $\bm{b}(t_{n},h)$, $\bm{d}(t_{n},h)$, and
$\bm{A}(t_{n},h)$ are independent of $t_n$
for the so-called class of {\em separable} basis functions $\{u_i\}_{i=1}^s$, 
i.e., if we have
(see~\cite[Section 3]{nguyen14}) 
$$
\bm{u}(x+y) = \mbox{\boldmath$\cal F$}(y)\bm{u}(x), \quad \forall x,y \in \IR
$$
where $\bm{u} = (1,t, u_1, ..., u_s)^T$ and
\mbox{\boldmath$\cal F$} is a square matrix of size $(s+2)\times(s+2)$ with entries that 
are univariate functions. This class includes
the most common functions such as polynomial, exponential, and trigonometric 
functions. In this particular class, the stability region
of the method for a given stepsize $h$ can be defined as
\begin{equation}
\label{stabregion}
S_h := \{z\in (-\infty,0]: \rho(\bm{M}_h(z))\leq 1\}
\end{equation}
where $\rho(\bm{M}_h(z))$ denotes the spectral radius of $\bm{M}_h(z)$ and 
\begin{equation}
\label{ae1}
\bm{M}_{(h)}(z)=
\begin{pmatrix}
z\big[\bm{A}(h) + \bm{e}\bm{b}^T(h) + \bm{c} \bm{d}^T(h)\big]&\bm{e} & \bm{e}+\bm{c}\\
z\bm{b}^T(h) &1 & 1\\
z\bm{d}^T(h)  &0 & 1\\
\end{pmatrix}.
\end{equation}
It becomes possible to characterize the stability region 
numerically or even analytically for some special cases
(cf.~\cite{nguyen1}, \cite{nguyen3}, \cite{nguyenfeptrk}). Furthermore, if we expand the 
coefficients into Taylor series
\begin{align*}
a_{ij}&=a_{ij}^{(0)}+a_{ij}^{(1)}h+a_{ij}^{(2)}h^2+...,\\
b_i&=b_i^{(0)}+b_i^{(1)}h+b_i^{(2)}h^2+...,
\end{align*}
then it can be shown that the leading terms are constant and conform to 
the coefficients of the conventional EPTRKN method defined by $\bm{c}$ (cf. \cite{Ozawa2}). 
Thus, the coefficients $\bm{b}(t,h),\bm{d}(t,h),\bm{A}(t,h)$ converge to those of the 
conventional EPTRKN method when $h\to 0$. Consequently, 
the amplification matrix $\bm{M}_{(t_n,h)}(z)$ and the stability region of a FEPTRKN method converge to  
the amplification matrix $\bm{M}(z)$ and the stability region of the corresponding conventional EPTRKN method, respectively.

\section{Extensions}
\label{sec:extensions}

We now discuss some aspects that are essential to producing competitive
numerical codes.

\subsection{Variable stepsize}
\label{variablestepsize}

When the stepsize $h_n$ is accepted and the next stepsize $h_{n+1}$ 
is suggested, 
the values in the next step are computed by
\begin{equation}
\label{variEPTRKN}
\begin{split}
\bm{Y}_{n+1} &= \bm{e}y_{n+1} + h_{n+1}\bm{c}y'_{n+1}+
h_{n+1}^2\bm{A}(t_{n},h_{n},h_{n+1})f(\bm{e}t_{n}+\bm{c}h_n,\bm{Y}_{n}),\\
y_{n+2} &= y_{n+1} + h_{n+1}y'_{n+1} +h^2_{n+1}\bm{b}^T(t_{n+1},h_{n+1})f(\bm{e}t_{n+1}+\bm{c}h_{n+1},\bm{Y}_{n+1}),\\
y'_{n+2} &= y'_{n+1}  +h_{n+1}\bm{d}^T(t_{n+1},h_{n+1})f(\bm{e}t_{n+1}+\bm{c}h_{n+1},\bm{Y}_{n+1}).
\end{split}
\end{equation}
The coefficients $\bm{b}^T(t_{n+1},h_{n+1})$ and $\bm{d}^T(t_{n+1},h_{n+1})$ are computed from the first two equations in \eqref{sysABeq} with $t=t_{n+1}$ and $h=h_{n+1}$. 
However, we cannot solve for $\bm{A}(t_{n},h_{n},h_{n+1})$ from the third equation in \eqref{sysABeq} if $h_{n+1}\not=h_n$. 
In particular, 
the coefficients
 $\bm{A}(t_{n},h_{n},h_{n+1})$ are 
obtained by solving the systems
\begin{align*}
\Big{(}\bm{v}_1(t_{n},h_{n},h_{n+1}), ...,
\bm{v}_s(t_{n},h_{n},h_{n+1}) \Big{)} 
&= h_{n+1}^2\bm{A}(t_{n},h_{n},h_{n+1})\bm{F}(t_{n}, h_{n}),
\end{align*}
where
$$
\bm{v}_i(t_{n},h_{n},h_{n+1}) = u_{i}(\bm{e}t_{n+1}+\bm{c}h_{n+1})-u_i(\bm{e}t_{n+1})-\bm{c}h_{n+1}u'_i(t_{n+1}),\qquad i=1,...,s.
$$

Note that even a conventional variable stepsize
EPTRKN induces variable coefficients. However, $\bm{b}$ and $\bm{d}$ remain constant
and the update for $\bm{A}$ leads to a simple diagonal scaling~\cite{cong01}. 
The next approximations $y_{n+2} \approx y(t_{n+1}+h_{n+1})$ and $y'_{n+2} \approx y'(t_{n+1}+h_{n+1})$ resulting from 
the new stepsize $h_{n+1}$ are subject to a local truncation error denoted by LTE 
which is often estimated by using another embedded method (see 
Section~\ref{subsec:embedded} below). If the estimated error LTE
is smaller than a prescribed tolerance TOL, then $h_{n+1}$ is accepted. 
Otherwise, it is rejected and a reduced stepsize $\tilde{h}_{n+1}$ is
used to recompute $\bm{Y}_{n+1}\approx 
y(\bm{e}t_{n+1}+\bm{c}\tilde{h}_{n+1})$, $y'_{n+2}\approx 
y'(t_{n+1}+\tilde{h}_{n+1})$, and $y_{n+2}\approx 
y(t_{n+1}+\tilde{h}_{n+1})$.
This process is repeated until an accepted value of $h_{n+1}$ is found.

Varying the stepsize from a collocation perspective means that, 
from the past accepted values $y_{n}$, $y'_{n}$, and $\bm{Y}_n$ we construct the 
collocation solution $u(t)$ defined in \eqref{interpol} and then evaluate and store the values $y_{n+1}=u(t_n+h_n)$, $y'_{n+1}=u'(t_n+h_n)$. The next stage values $\bm{Y}_{n+1}$  
are generated from this collocation function using an estimated stepsize $h_{n+1}$ (see Section \ref{errorcontrol} below). 
The acceptance of this $h_{n+1}$ is subject to a local truncation error for computing $y_{n+2}$ using this stepsize (cf. Section \ref{errorcontrol}). 
Therefore, the collocation function $u(t)$ and values $y_{n+1}$, $y'_{n+1}$ remain the same when $h_{n+1}$ varies. We only
adjust the stepsize $h_{n+1}$ for computing $\bm{Y}_{n+1}$, specifically, 
$\bm{Y}_{n+1} = u(\bm{e}t_{n+1}+\bm{c}h_{n+1})$, see \eqref{iteratepoly}. 
As a consequence, the following generalization 
of~\cite[Theorem~2.1]{cong01} holds.

\begin{thm}
The $s$-stage variable stepsize
FEPTRKN method \eqref{variEPTRKN} is of stage order $p=s$ and of step order at 
least $p=s$ for any set of distinct collocation points $(c_i)_{i=1}^s$. It has step order $p=s+q$, $q=1,2$, if the $(c_i)_{i=1}^s$ satisfy
the orthogonality conditions
$$
\int_0^1\xi^k\prod_{i=1}^s(\xi-c_i)d\xi = 0,\qquad k=0,q-1.
$$
\end{thm}


\subsection{Interpolation}
\label{interpolation}

A continuous extension of an $s$-stage FEPTRKN method determined by $(\bm{c},\bm{A}(t_n,h_n),\bm{b}(t_n,h_n),\bm{d}(t_n,h_n))$ is defined  
as follows (cf. \cite{cong01}, \cite{nguyenfeptrk})
\begin{align}
\label{denseoutput}
y_{n+\xi} &= y_{n} + \xi h_ny'_n+ (\xi h_n)^2\bm{b}^T(t_n,h_n,\xi)
f(\bm{e}t_{n}+\bm{c}h,\bm{Y}_{n}),\\
\label{denseoutput1}
y'_{n+\xi} &= y'_{n} + \xi h_n\bm{d}^T(t_n,h_n,\xi)
f(\bm{e}t_{n}+\bm{c}h,\bm{Y}_{n}),\qquad 0\leq \xi \leq 1,
\end{align}
where $y_{n+\xi}\approx y(t_{n+\xi})$ and $y'_{n+\xi}\approx y'(t_{n+\xi})$ with 
$t_{n+\xi}:=t_n+\xi h_n$, and the coefficients 
$\bm{b}(t_n,h_n,\xi)$ and $\bm{d}(t_n,h_n,\xi)$ are obtained from the equations
\begin{equation}
\label{densepara}
\begin{split}
\Big{(}u_{1}(t_{n+\xi})-u_1(t_{n})-\xi h_nu'_1(t_{n}), ..., 
u_{s}(t_{n+\xi})-u_s(t_{n})-\xi h_nu'_s(t_{n})\Big{)} 
&= (\xi h_n)^2\bm{b}^T(t_n,h_n,\xi)\bm{F}(t_n, h_n),\\
\Big{(}u_{1}(t_{n+\xi})-u_1(t_{n}), ..., 
u_{s}(t_{n+\xi})-u_s(t_{n})\Big{)} 
&= \xi h_n\bm{d}^T(t_n,h_n,\xi)\bm{F}(t_n, h_n).
\end{split}
\end{equation}

In other words, any $y_{n+\xi}$ and $y'_{n+\xi}$ are retrieved as if $t_{n+\xi}=t_n+\xi h_n$
was the end point. From their definitions one has 
$y_{n+\xi}=u(t_n+\xi h_n)$ and $y'_{n+\xi}=u'(t_n+\xi h_n)$, where 
$u(t)$ is the collocation solution defined in 
\eqref{interpol}. From \eqref{intererror} it is easy to check that
$y_{n+\xi}-y(t_n+\xi h_n)=O(h_n^{s+2})$ and
$y'_{n+\xi}-y'(t_n+\xi h_n)=O(h_n^{s+1})$, $\forall \xi \in [0,1]$.
Hence, the following result holds.

\begin{thm} The FEPTRKN method defined by \eqref{denseoutput}--\eqref{denseoutput1} with 
$\bm{b}$ and $\bm{d}$ defined by \eqref{densepara} produces a continuous 
FEPTRKN method of order $s$, i.e.,
$$
y(t_n+\xi h_n)-y_{n+\xi}=O(h_n^{s+2}),\quad y'(t_n+\xi h_n)-y'_{n+\xi}=O(h_n^{s+1}),\qquad 0 \le \xi \le 1.
$$
\end{thm}


It can further be shown that when the collocation functions are 
the monomials, the coefficient $\bm{b}(t_n,h_n,\xi)$ depends 
only on $\xi$ (cf.~\cite{cong01}).

\subsection{Embedded methods}
\label{subsec:embedded}

\def\tc{\tilde{c}}
\def\ts{{\tilde{s}}}
\def\tY{\tilde{Y}}
\def\ty{\tilde{y}}
\def\tu{\tilde{u}}
\def\tbm#1{\tilde{\bm{#1}}}

Consider an $s$-stage FEPTRKN method with coefficients $(\bm{c},\bm{A},\bm{b},\bm{d})$.
We now wish to have an embedded FEPTRKN method ($\tbm{c},\tbm{A},\tbm{b},\tbm{d})$ paired with the FEPTRKN method $(\bm{c},\bm{A},\bm{b},\bm{d})$ 
to cheaply estimate the errors and control 
the stepsizes in practice.

Let $\ts < s$, $\{\tc_1, ..., \tc_\ts\}
\varsubsetneq \{c_1, ..., c_s\}$,
$\tbm{c} = (\tc_1, ..., \tc_\ts)^T$
and $\tbm{e} = (1, ..., 1)^T$ of length $\ts$. 
An embedded pair FEPTRKN methods in which 
another approximation $\ty_{n+1}$ to $y(t_{n+1})$ can 
be computed without extra function evaluations is defined by 
\begin{equation}
\label{embpair}
\begin{split}
y_{n+1} &= y_{n} +h_ny'_n+ h_n^2\bm{b}^T(t_n,h_n)f(\bm{e}t_{n}+\bm{c}h_n,\bm{Y}_{n}),\\
y'_{n+1} &= y'_{n} + h_n\bm{d}^T(t_n,h_n)f(\bm{e}t_{n}+\bm{c}h_n,\bm{Y}_{n}),\\
\ty_{n+1} &= y_{n} +h_ny'_n + h_n^2\tbm{b}^T(t_n,h_n)f(\tbm{e}t_{n}+\tbm{c}h_n,\tbm{Y}_{n}),\\
\bm{Y}_{n+1} &= \bm{e}y_{n+1} +\bm{c}h_ny'_{n+1} +  h_n^2\bm{A}(t_n,h_n)f(\bm{e}t_{n}+\bm{c}h_n,\bm{Y}_{n})
\end{split}
\end{equation}
where 
$\tbm{Y}_{n}=(\tY_{n,1},...,\tY_{n,\ts})$ 
and 
$f(\tbm{e}t_{n}+\tbm{c}h_n,\tbm{Y}_{n})
= 
(f(t_n+\tc_1h_n,\tY_{n,1}), ..., f(t_n+\tc_\ts h_n,\tY_{n,\ts}))^T 
$
are defined according to the mapping (which is well defined
because the $c_i$ are distinct)
\begin{align*}
\text{if}\quad \tc_i=c_j\quad \text{then}\quad 
\tY_{n,i}=Y_{n,j}.
\end{align*}
The coefficient $\tbm{b}(t_n,h_n)$ is obtained by solving the linear system
\begin{align*}
\Big{(}\tu_{1}(t_n+h_n)-\tu_1(t_{n}) -h_n\tu'_1(t_{n}), ..., 
\tu_{\ts}(t_{n}+h_n)-\tu_\ts(t_{n})-h_n\tu'_\ts(t_{n})\Big{)} 
&= h_n^2\tbm{b}^T(t_n,h_n)\tbm{F}(t_n, h_n)
\end{align*}
where
\begin{align*}
\tbm{F}(t_n,h_n)&=
\bigg{(}\tu''_{1}(\tbm{e}t_n+\tbm{c}h_n),
...,
\tilde{u}''_{\ts}(\tbm{e}t_n+\tbm{c}h_n)\bigg{)}.
\end{align*}

In other words, $\tbm{b}$ is the coefficient of the FEPTRKN method 
generated from $\tbm{c}$ and the basis of functions 
$\{\tu_1, ..., \tu_\ts\} \varsubsetneq \{u_1, ..., u_s\}$.
The solution $\ty_{n+1}$ is computed by combining the subset 
of past stage values with indices corresponding to $\tbm{c}$.
For this definition of $\tbm{b}$ we can assert that 
$\ty_{n+1}=y(t_{n+1})+O(h^{\ts+1})$ as a result of 
Theorem~\ref{genorder}.

\begin{thm}
An $s$-stage embedded pair FEPTRKN \eqref{embpair} 
produces solutions $y_{n+1}$ and $\ty_{n+1}$ that satisfy
$$
y_{n+1}-\ty_{n+1}=O(h^{\ts+1}),
$$
for all set of distinct collocation parameters $(c_i)_{i=1}^s$.
\end{thm}


\subsection{Error control and step-size change}
\label{errorcontrol}

Let us discuss a strategy for changing stepsizes in the implementation of 
a FEPTRKN method of order $p$ embedded with a FEPTRKN method of order $\tilde{p}<p$ using the variable stepsize technique in Section \ref{variablestepsize}.
 At each step we compute a local truncation error LTE as follows
\begin{equation}
\label{LTE1}
\text{LTE} = \|y_{n+1} - \tilde{y}_{n+1}\| = O(h^{\tilde{p}+1}). 
\end{equation}
Formula \eqref{LTE1} for computing $\text{LTE}$ is simpler than the one in \cite{cong01} where $\text{LTE}$ is computed by
\begin{equation}
\label{LTE2}
\text{LTE} = \sqrt{\|y_{n+1} - \tilde{y}_{n+1}\|^2 + \|y'_{n+1} - \tilde{y}'_{n+1}\|^2}= O(h^{\tilde{p}+1}). 
\end{equation}
Here $\tilde{y}'_{n+1}$ is computed by
$$
\ty'_{n+1} = y'_{n} + h_n\tbm{d}^T(t_n,h_n)f(\tbm{e}t_{n}+\tbm{c}h_n,\tbm{Y}_{n}),
$$
where $\tbm{Y}_{n}$ is defined as in Section \ref{subsec:embedded} and  ($\tbm{c},\tbm{A},\tbm{b},\tbm{d})$ are the coefficients of the embedded method.  

In our experiments we have found out that using \eqref{LTE2} instead of \eqref{LTE1} for computing $\text{LTE}$ results in having smaller stepsizes, and, therefore, yields smaller errors. However, having smaller stepsizes makes the FEPTRKN methods use more function evaluations. From numerical experiments, we do not see any advantage of using \eqref{LTE2} over using \eqref{LTE1} for computing $\text{LTE}$.   

A stepsize $h_n$ is accepted if $\text{LTE}\le\text{TOL}$ and rejected if otherwise. 
If $h_n$ is rejected, then we repeat the computations with the new stepsize $h_n=h_n/2$ until an 
accepted $h_n$ is found. 
If $h_n$ is accepted, then the stepsize $h_{n+1}$ in the next step is chosen as 
$$
h_{n+1} = h_n \min\bigg\{2,\max\big\{0.5,0.8\big(\text{TOL}/\text{LTE}\big)^{1/(\tilde{p}+1)}\big\}\bigg\}. 
$$
It follows from this formula that the ratio $\frac{h_{n+1}}{h_n}$ always stays in the interval $[0.5,2]$.  
This stepsize changing technique was also used in \cite{cong01}. 

\section{Numerical experiments}

The experiments were conducted in double precision (machine precision = $0.2\times 10^{-15}$) using MATLAB on a PC computer with 2Gb RAM and 2.8 GHz.

\subsection{Derivation of some methods}
\label{derivation}

We derived some methods with the following bases:
\begin{enumerate}
\item{\{$x^2,x^3,x^4$\} for eptrkn52.}
\item{\{$x^2,\cos(\omega x),\sin(\omega x)$\} for feptrkn52.}
\item{\{$x^2,x^3,x^4,x^5$\} for eptrkn73.}
\item{\{$\cos(\omega x),\sin(\omega x),\cos(2\omega x),\sin(2\omega x)$\} for feptrkn73.}
\item{\{$x^2,x^3,x^4,x^5,x^6$\} for eptrkn84.}
\item{\{$x^2,\cos(\omega x),\sin(\omega x),\cos(2\omega x),\sin(2\omega x)$\} for feptrkn84.}\item{\{$x^2,x^3,x^4,x^5,x^6,x^7$\} for eptrkn95.}
\item{\{$\cos(\omega x),\sin(\omega x),\cos(2\omega x),\sin(2\omega x),\cos(3\omega x),\sin(3\omega x)$\} for feptrkn95.}
\end{enumerate}
The coefficients of FEPTRKN methods based on bases 2, 4, 6, and 8 above are independent of $t$ but they are  functions of $\omega h$. The coefficients of EPTRKN methods, i.e., methods based on bases 1, 3, 5, and 7 above, are constants. 

Methods based on bases 1 and 2 are 3-stage methods and are implemented with
$$
\bm{c} = \begin{bmatrix}
0.18677613705141
   &0.75202972313575
   &1.66119413981284
\end{bmatrix}^T. 
$$
This set of collocation points $(c_i)_{i=1}^3$ are determined to satisfy orthogonal conditions in \eqref{eq17} and the equation 
$$\int_0^2\prod_{i=1}^3(\xi-c_i)\, d\xi =0.$$ 
These methods have the same step order $p=5$ by Theorem \ref{super2} and are denoted by eptrkn52 for basis 1 and feptrkn52 for basis 2. Their embedded methods have the same step order $p=2$. 

Methods based on bases 3 and 4 are 4-stage methods. We implement these methods with 
$$
\bm{c} = \begin{bmatrix}
10027252023777 
   &0.46050359576754
   &0.86389485661306
  & 1.43247188452449
\end{bmatrix}^T
$$
which is computed to satisfy orthogonality conditions in \eqref{eq20}.
These methods share the same accuracy order $p=8$ by Theorem \ref{super3} and are denoted by eptrkn73 and feptrkn73  for basis 3 and basis 4, respectively. Their embedded methods are 3-stage methods and have the same accuracy order $p=3$. 

Methods based on bases 5 and 6 are 5-stage methods. We implement these methods with 
$$
\bm{c} = 
\begin{bmatrix}
0.0911311145011
   &0.4288524464674
   &0.8402456535427
   &1.3131095250315
   &1.8405501493461
\end{bmatrix}^T
$$
which is computed to satisfy orthogonality conditions in \eqref{eq20} and the equation
$$\int_0^2\prod_{i=1}^5(\xi-c_i)\, d\xi =0.$$ 
These methods have the same accuracy order $p=8$ by Theorem \ref{super3} and are denoted by eptrkn84 for basis 5 and feptrkn84 for basis 6. Their embedded methods are 4-stage methods and have the same accuracy order $p=4$. 

Methods based on bases 7 and 8 are 6-stage methods. These methods are implemented with 
$$
\bm{c} = 
\begin{bmatrix}
0
&0.15981788694649
   &0.47315766336506
  & 0.80767247891979
   &1
   &1.55935197076839
\end{bmatrix}^T
$$
which is computed so that orthogonality conditions in \eqref{eq20} hold and that the set $(c_i)_{i=1}^s$ contains two values 0 and 1. Having 0 and 1 in $(c_i)_{i=1}^s$ helps reducing one right-hand side function evaluation per step by reusing $f(t_{n+1},y_{n+1})$ computed from the previous step. 
These methods share the same accuracy order $p=9$ by Theorem \ref{super3} and are denoted by eptrkn95 for basis 7 and feptrkn95 for basis 8. Their embedded methods are 5-stage methods and have the same accuracy order $p=5$.

\subsection{Stability regions}

As mentioned earlier, the coefficients of FEPTRKN methods converge to those of the corresponding EPTRKN methods when $h$ goes to zero. 
Thus, the stability regions of FEPTRKN methods converge to those of the corresponding EPTRKN methods when $h$ approaches to zero. 

Using equations \eqref{stabregion}--\eqref{ae1} we plot the stability regions of the feptrkn52, feptrkn73, feptrkn84, and feptrkn95 methods with constant stepsizes, i.e., $h_n=h=const$. Figure \ref{figstab1} plots the stability regions of the feptrkn52 method for $\omega h\in [0,5]$ (left) and of the feptrkn73 method for $\omega h\in [0,4]$ (right). 
\begin{figure}
\centerline{%
\includegraphics[scale=0.85]{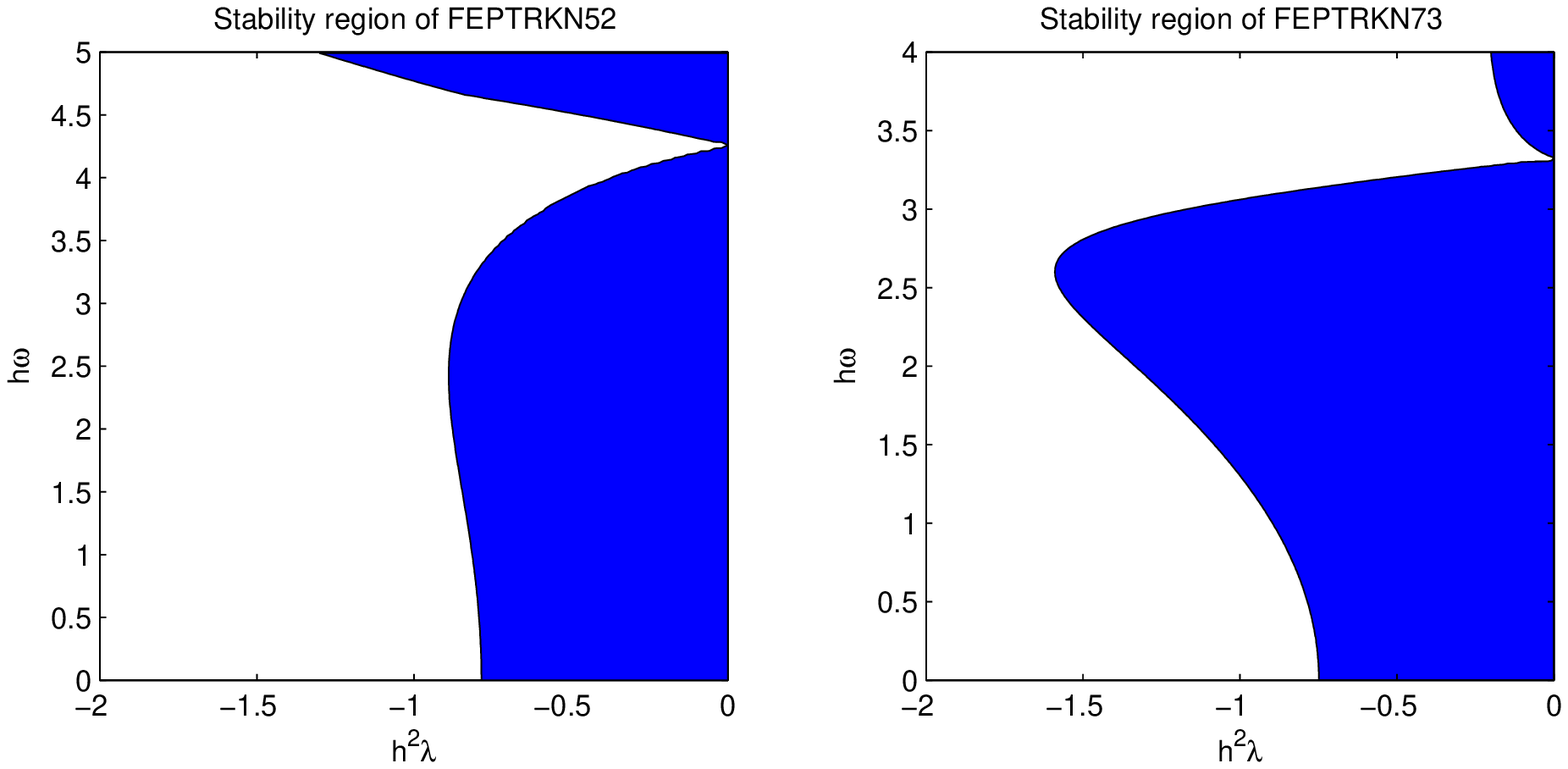}}
\caption{Plots of stability regions of the feprtkn52 and feptrkn73 methods.}
\label{figstab1}
\end{figure}
It can be seen from the left figure in Figure \ref{figstab1} that the stability region of the feptrkn52 method is greater than the stability region of the eptrkn52 method for $\omega h\in [0,3]$. 
From the right figure in Figure \ref{figstab1} one concludes that the stability region of the feptrkn74 method is larger than that of the eptrkn74 method when $\omega h\in (0,2.7]$. To have large stability regions, one should not use $\omega h>3.5$ and $\omega h>3$ for the feptrkn52 and feprtkn73 methods, respectively. 

Figure \ref{figstab2} plots the stability regions of the feptrkn84 method (left) and the feptrkn95 method (right) for $\nu=\omega h\in [0,4]$. 
\begin{figure}
\centerline{%
\includegraphics[scale=0.85]{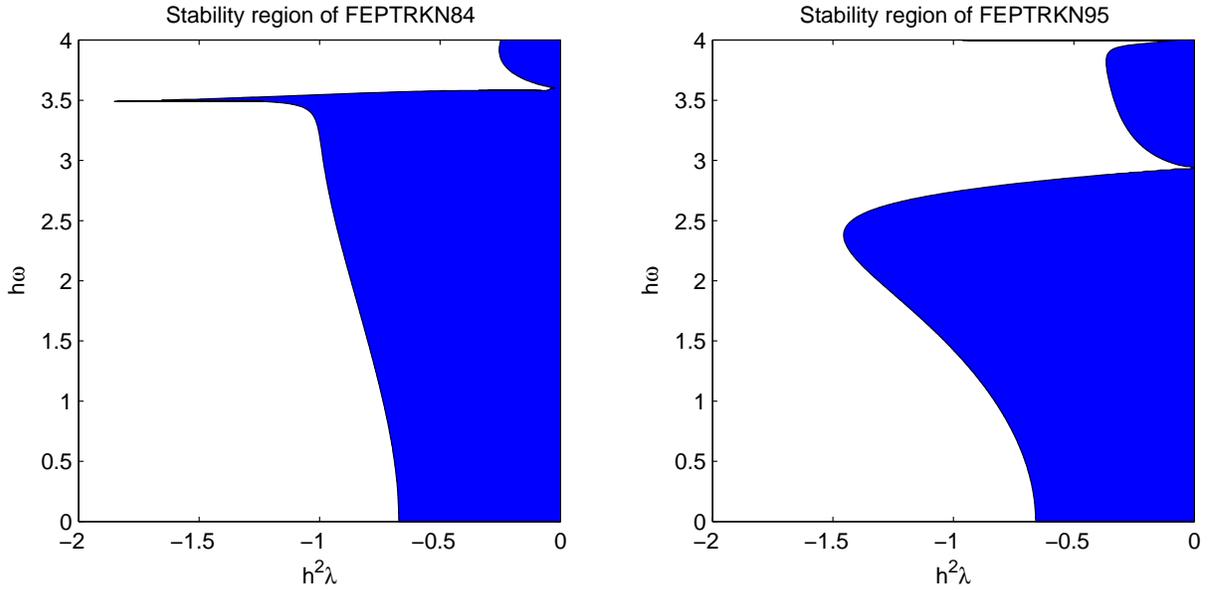}}
\caption{Plots of stability regions of the feptrkn84 and feptrkn95 methods.}
\label{figstab2}
\end{figure}
It can be seen from the left figure in Figure \ref{figstab2} that the stability region of the feptrkn84 method for $\omega h\in [0,3.4]$ is larger than that of the eptrkn84 method. 
The right figure in Figure \ref{figstab2} indicates that the stability region of the feptrkn95 method when $\omega h\in [0,2.5]$ is larger than that of the eptrkn95 method. Figure \ref{figstab2} suggests that to have large stability regions one should not use $\omega h>3.5$ and $\omega h>2.8$ for the feptrkn84 method and the feptrkn95 method, respectively.  

Although the stability regions of (F)EPTRKN methods are not as large as those of Runge-Kutta-Nystr\"{o}m methods, they are still sufficiently large for solving nonstiff problems (cf. \cite{cong99}, \cite{cong01}). 

\subsection{Test problems}

Numerical experiments are carried out with the following problems:

{\bf BETT}-Consider the system (see \cite[p.141]{simos})

\begin{align*}
y_1'' &= -y_1 + 0.001 \cos t,\qquad y_1(0) = 1,\qquad y_1'(0) = 0,\\
y_2'' &= -y_2 + 0.001 \sin t,\qquad y_2(0) = 0,\qquad y_2'(0) = 0.9995,
\end{align*}
where the integration domain is $[0,40]$. 
The solution to the problem is given by
\begin{align*}
y_1(t) = \cos t + 0.0005t\sin t,\\
y_2(t) = \sin t - 0.0005t\cos t.
\end{align*}

{\bf NEWT}-The two-body gravitational problem
\begin{align*}
y_1'' &= -\frac{y_1}{(y_1^2 + y_2^2)^{3/2}},\qquad y_1(0) = 1-e,\qquad y_1'(0) = 0, \\
y_2'' &= -\frac{y_2}{(y_1^2 + y_2^2)^{3/2}},\qquad y_2(0) = 0,\qquad y_2'(0) = \sqrt{\frac{1+e}{1-e}},
\end{align*}
whose exact solution is given by
\begin{align*}
y_1(t) &= \cos u - e,\\
y_2(t) &= \sqrt{1-e^2}\sin u,
\end{align*}
where $u$ is the solution of Keppler's equation $u=t+e\sin u$. The integration domain for this problem is $[0,20]$.

\subsection{Results and discussion}

Table \ref{table1} presents numerical results for solving the BETT problem using the four methods eptrkn52, eptrkn73, eptrkn84, and eptrkn95 implemented with constant stepsizes. The NCD values in Table \ref{table1} are computed by the formula
$$
\text{NCD} = \log_{10} \max(E_1,E_2),\qquad E_i:=\max_{0\le t_n\le T} |y_{i}^{comput}(t_n) - y_i(t_n)|.
$$
The NCD values can be used to determine accuracy order of a method in practice. In particular, if a method has an accuracy order $p$, then the error for computing $y_i(t_n)$ is of the form $Ch^p$. 
Thus, the NCD values for computing $y_i(t_n)$ is about $\log_{10}C + p\log_{10}h$. Therefore, the difference between the NCD values at stepsize $h$ and at stepsize $h/2$ is about $p\log_{10}2\approx 0.3p$. 
Thus, one can estimate $p$ the accuracy order of the method by finding differences between NCD values at stepsize $h$ and $h/2$ and then dividing these differences by $0.3$. 

By subtracting the NCD values at stepsize $h/2$ from the NCD values at stepsize $h$ and then dividing the results by $0.3$ one can see that the accuracy orders of the eptrkn52, eptrkn73, eptrkn84, and eptrkn95 methods are $p=5$, $p=7$, $p=8$, and $p=9$, respectively. This agrees with the super-convergence results in Theorems \ref{super2} and \ref{super3}. 
We also see from Table \ref{table1} that the higher accuracy order of the methods, the more accurate they are for the same stepsizes $h$. However, higher order methods use more function evaluations in each step. Thus, the data in Table \ref{table1} does not tell us which method is the best in this experiment. The main purpose of using the results in Table \ref{table1} is to verify the super-convergence results in Theorems \ref{super2} and \ref{super3}.

\begin{table}[ht] 
\caption{NCD values for the BETT problem}
\label{table1}
\centering
\small
\renewcommand{\arraystretch}{1.2}
\begin{tabular}{@{}l@{}|c@{\hspace{2mm}}|c@{\hspace{2mm}}|c@{\hspace{2mm}}|c@{\hspace{2mm}}|c@{\hspace{2mm}}|
c@{\hspace{2mm}}c@{\hspace{2mm}}|}
\hline
&$h$	&eptrkn52 & eptrkn73& eptrkn84 &eptrkn95 \\
\hline
&$1/2^1$  &-2.6  &-4.0  &-6.0  &-5.9\\ 
&$1/2^2$  &-4.1  &-6.3  &-8.2  &-8.7\\ 
&$1/2^3$  &-5.7  &-8.7  &-10.8  &-11.7\\ 
&$1/2^4$  &-7.2  &-11.1  &-13.5  &-14.6\\ 
&$1/2^5$  &-8.7  &-13.5  &-15.1  &-14.3\\ 
&$1/2^6$  &-10.2  &-15.5  &-15.7  &-14.5\\ 
&$1/2^7$  &-11.7  &-14.7  &-14.4  &-14.7\\ 
&$1/2^8$  &-13.2  &-14.3  &-14.3  &-14.4\\ 
&$1/2^9$  &-14.5  &-14.6  &-14.5  &-14.9\\ 
\hline
\end{tabular} 
\end{table}  

Table \ref{table2} presents numerical results for solving the NEWT problem with the four EPTRKN methods eptrkn52, eptrkn73, eptrkn84, and eptrkn95. Again by subtracting the NCD values at stepsize $h/2$ from the NCD values at stepsize $h$ and then dividing the obtained results by 0.3,  we can see that the accuracy orders of the eptrkn52, eptrkn73, eptrkn84, and eptrkn95 methods are $p=5$, $p=7$, $p=8$, and $p=9$, respectively.

\begin{table}[ht] 
\caption{NCD values for the NEWT problem, $e=0.01$.}
\label{table2}
\centering
\small
\renewcommand{\arraystretch}{1.2}
\begin{tabular}{@{}l@{}|c@{\hspace{2mm}}|c@{\hspace{2mm}}|c@{\hspace{2mm}}|c@{\hspace{2mm}}|c@{\hspace{2mm}}|
c@{\hspace{2mm}}c@{\hspace{2mm}}|}
\hline
&$h$	&eptrkn52 & eptrkn73& eptrkn84 &eptrkn95 \\
\hline
&$1/2$ &-0.9  &-2.2 	&-2.6 		&-2.9\\
&$1/2^2$ &-2.4  &-4.5 	&-6.2 		&-6.0\\
&$1/2^3$ &-3.9  &-6.9 	&-8.9 		&-9.2\\
&$1/2^4$ &-5.4  &-9.2 	&-11.5 	&-12.1\\
&$1/2^5$ &-6.9  &-11.5 	&-13.6    &-13.7\\
&$1/2^6$ &-8.4  &-12.6 	&-13.7    &-13.3\\
&$1/2^7$ &-9.9  &-12.8 	&-13.1    &-12.8\\
&$1/2^8$ &-11.4 &-12.9  &-13.3   & -12.6\\
&$1/2^9$ &-13.1 &-12.4  &-12.5    &-12.5\\
\hline
\end{tabular} 
\end{table}  

We also carried out similar experiments for the four FEPTRKN methods feptrkn52, feptrkn73, feptrkn83, and feptrkn95. We found out from our experiments that the accuracy orders of the feptrkn52, feptrkn73, feptrkn83, and feptrkn95 methods are $p=5$, $p=7$, $p=8$, and $p=9$, respectively. Again, this agrees with our super-convergence results in Theorems \ref{super2} and \ref{super3}. For simplicity we do not include the numerical results of these experiments in this paper. 

Figure \ref{figBETT} and Figure \ref{figNEWT} present the numerical results for solving 
the BETT and NEWT problems using the eight EPTRKN and FEPTRKN methods derived in Section \ref{derivation}. 
We also include numerical results obtained by the MATLAB function ode45. This function is a MATLAB implementation of the Runge-Kutta method DOPRI45 (see, e.g., \cite{Hairer}). 
The eight EPTRKN and FEPTRKN methods are implemented with variable stepsizes technique 
as described in Section \ref{variablestepsize}. In Figures \ref{figBETT} and \ref{figNEWT} we denote by NFE the number times evaluating $f(t_n+c_ih_n,Y_{n,i})$ of each of these methods. 
Note that the main computation cost for solving large-scale equation \eqref{sys} is evaluating the right-hand side function $f$.
The errors reported in Figures \ref{figBETT} and \ref{figNEWT} are computed as follows
$$
\text{Error} = \sqrt{\big(y^{\text{comput}}_1(t_{\text{end}}) - y_1(t_{\text{end}})\big)^2 + \big(y^{\text{comput}}_2(t_{\text{end}}) - y_2(t_{\text{end}})\big)^2}. 
$$

Figure \ref{figBETT} plots the numerical results for the BETT problem. 
The left figure in Figure \ref{figBETT} plots the numerical results obtained by the methods eptrkn52, eptrkn73, eptrkn84, eptrkn95, and the ode45 function. 
One can see from this figure that all the EPTRKN methods are better than the ode45 method. 
The right figure in Figure \ref{figBETT} plots the results obtained the eight methods derived in Section \ref{derivation} and ode45. Also, it can be seen from the right figure in Figure \ref{figBETT} that all of the FEPTRKN methods are much better than their corresponding EPTRKN methods in this experiment.

\begin{figure}
\centerline{%
\includegraphics[scale=0.85]{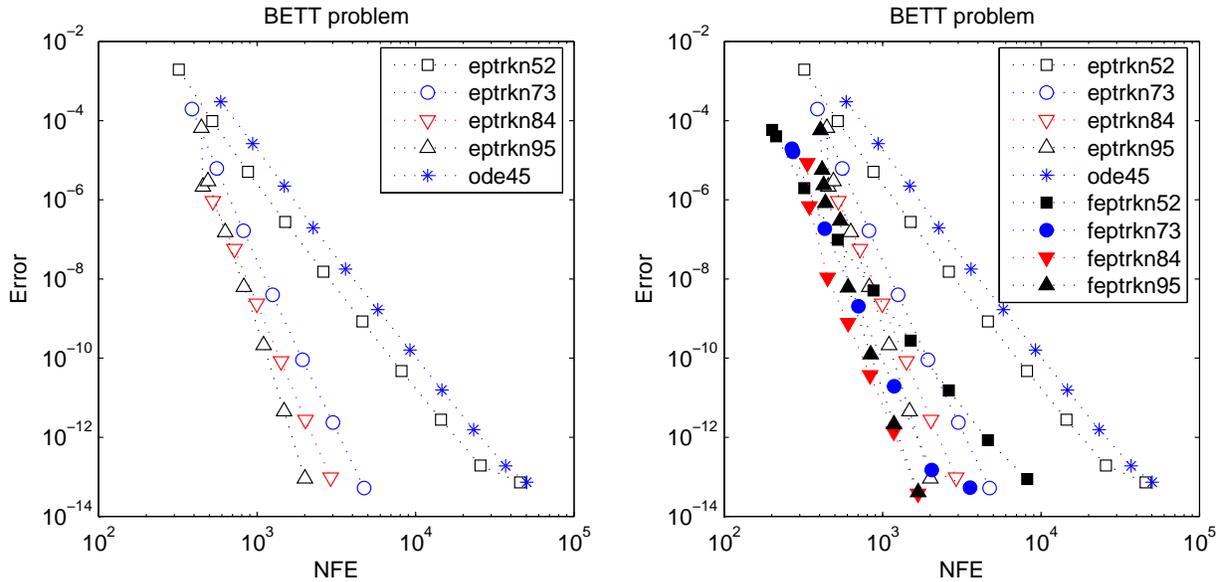}}
\caption{Numerical results for the BETT problem.}
\label{figBETT}
\end{figure}

Figure \ref{figNEWT} plots the results for the NEWT problem. 
From the left figure in Figure \ref{figNEWT} we can see that the eptrkn52 method is slightly better than the ode45 method. The other EPTRKN methods are much better than the ode45 method. The feptrkn95 method yields the best result for small TOL.  
The right figure in Figure \ref{figNEWT} plots the results for the eight methods derived in Section \ref{derivation} and the ode45 function. Again, we see that all functionally-fitted EPTRKN methods are superior to the corresponding EPTRKN methods. All of the methods derived in Section \ref{derivation} except for the eptrkn52 are much better than the ode45 method when TOL is small. 

Note that all computations here is done in sequential computing environment. 
The new methods will be much more efficient for solving large-scale nonstiff problems 
when implemented on shared memory computers as they require only one right-hand side function evaluation per step in parallel computing environments (see, e.g., \cite{cong00}). 

\begin{figure}
\centerline{%
\includegraphics[scale=0.85]{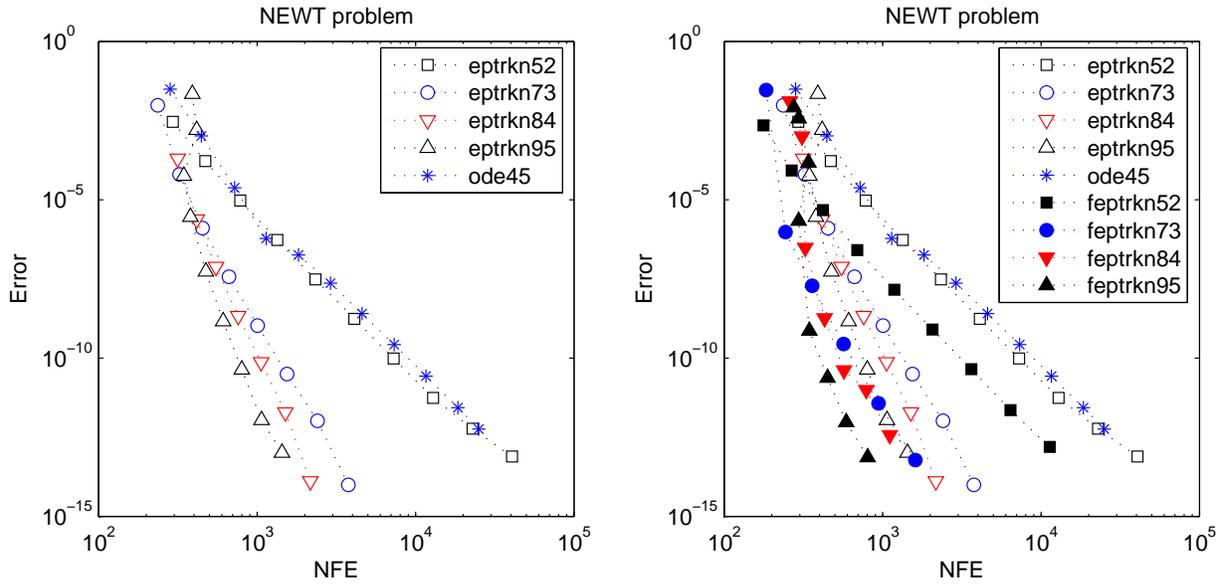}}
\caption{Numerical results for the NEWT problem when $e=0.01$.}
\label{figNEWT}
\end{figure}

\section{Concluding remarks}

A new class of functionally fitted explicit pseudo two-step RKN (FEPTRKN) methods has been developed and  studied in this paper.
The advantage of functionally fitted methods over classical methods is that we can fine-tune the choice of bases to have methods which can better capture the special properties of a particular problem that may be known in advance. 
 We have proved a new superconvergence result that allows us to construct $s$-stage (F)EPTRKN methods having accuracy orders $p=s+3$ and have shown that this new super-convergence result can be realized in practice.  
If we choose the monomials as the basis functions, then we recover EPTRKN methods. 
Moreover, we can obtain a larger class of methods 
by choosing various basis functions such as exponential functions, trigonometric polynomials, or mixed algebraic and trigonometric polynomials.
Important aspects to producing robust and efficient numerical codes, namely, variable stepsizes, error estimation via embedded methods, as well as continuous extensions are also considered. 
While it is primarily designed for parallel computers, numerical comparisons with the MATLAB function ode45  
in the paper show that the new methods can work better than one of the most commonly used methods on nonstiff problems in which accuracy (rather than stability) controls stepsizes.  
The new methods generally have a competitive advantage when the basis functions and the collocation parameters are chosen suitably. 
Since FEPTRKN methods have the structure of EPTRKN methods, they will be more efficient when implemented on parallel computers.

\end{document}

%% file: macros.tex
\newtheorem{remark}{Remark}[section]

\newenvironment{proof}{\begin{trivlist}
                       \item[]{\bf Proof.}
                       \hspace{0cm}}{\hfill $\Box$
                       \end{trivlist}}
                     {\end{enumerate}}

\newcounter{statement}
\newcounter{algorithm}
\newcounter{secalgorithm}[section]
\newcounter{subalgorithm}[subsection]
{\catcode`\;=\active%
\gdef\pointvirgule{\catcode`\;=\active%
\def;{\unskip\kern 2pt\string;~~~\ignorespaces}}}
\def\refstepalgorithms{\stepcounter{subalgorithm}\stepcounter{secalgorithm}
\refstepcounter{algorithm}}
{\end{tabbing}}
\newenvironment{algorithm*}[1]{\pointvirgule\setcounter{statement}{0}
\def\N>{\stepcounter{statement}\thestatement.\>}\def\B>{\hskip0.1em\vrule\>}
\def\NN>{\stepcounter{statement}\ifnum\thestatement<10\phantom{0}\fi
\thestatement.\>}\def\[##1]{{\bf##1}}
\begin{tabbing}
++\=++\=++\=++\=++\=++\=++\=++\=++\=++\=++\=++\=++\=++\kill
{\sc Algorithm}: {#1}\\*[1mm]}%
{\end{tabbing}}
\newenvironment{*algorithm*}{\pointvirgule\setcounter{statement}{0}
\def\N>{\stepcounter{statement}\thestatement.\>}\def\B>{\hskip0.1em\vrule\>}
\def\NN>{\stepcounter{statement}\ifnum\thestatement<10\phantom{0}\fi
\thestatement.\>}\def\[##1]{{\bf##1}}
\begin{tabbing}
++\=++\=++\=++\=++\=++\=++\=++\=++\=++\=++\=++\=++\=++\kill
}{\end{tabbing}}
{\end{tabbing}}
\newenvironment{ualgorithm*}[1]{\pointvirgule\setcounter{statement}{0}
\def\N>{\stepcounter{statement}\thestatement.\>}\def\B>{\hskip0.1em\vrule\>}
\def\NN>{\stepcounter{statement}\ifnum\thestatement<10\phantom{0}\fi
\thestatement.\>}\def\[##1]{{\bf##1}}
\begin{tabbing}
++\=++\=++\=++\=++\=++\=++\=++\=++\=++\=++\=++\=++\=++\kill
\underline{{\sc Algorithm}}: {#1}\\*[1mm]}%
{\end{tabbing}}
\catcode`@=11
\def\@comment[#1]#2{$\{$\hbox to #1{#2 \dotfill}$\}$}
\def\comment{\@ifnextchar[{\@comment}{\@comment[10.5cm]}}

\catcode`@=12

\def\Span{\mathop{\rm Span}}

\def\eps{\ifmmode\epsilon\else$\epsilon$\fi}
\def\veps{\ifmmode\varepsilon\else$\varepsilon$\fi}

\def\IR{\mbox{\bf R}} 

\def\IR{\mbox{\rm I\hspace{-.15em}R}}

\def\adots{\mathinner{\mkern2mu\raise1pt\hbox{.}
\mkern3mu\raise4pt\hbox{.}\mkern1mu\raise7pt\hbox{.}}}

\def\og{\leavevmode\raise.3ex\hbox{$\scriptscriptstyle\langle\!\langle$}}
\def\fg{\leavevmode\raise.3ex\hbox{$\scriptscriptstyle\rangle\!\rangle$}}

\def\romain#1{\uppercase\expandafter{\romannumeral#1}}

\def\trait{\hskip 0.5mm \raise -6pt\vbox{\vrule height 18pt}\>}
\def\traitl{\hskip 0.5mm \raise -11pt\vbox{\vrule height 25pt}\>}
\def\hfl#1#2{\smash{\mathop{\hbox to 20mm {\rightarrowfill}}
\limits^{\scriptstyle #1}_{\scriptstyle #2}}}

\def\SS#1#2{{#1}\kern .15em{#2}}
\def\into#1{\S\kern .15em\ref{#1}}

\catcode`@=11
\newcount\@hcaptype
\def\@@hcaption[#1]#2#3{\par\@bsphack\vbox{\def\@captype{#1}\caption{#3}}
{\if@filesw{\xdef\@gtempa{\write\@auxout{\string\newlabel{#2}
{{\csname p@#1\endcsname\csname the#1\endcsname}{\thepage}}}}}}
\@gtempa\if@nobreak\ifvmode\nobreak\fi\fi\fi\@esphack}
\def\@hcaption#1#2{\ifnum\@hcaptype=0\@@hcaption[figure]{#1}{#2}
\else\@@hcaption[table]{#1}{#2}\fi}
\def\hcaption{\@ifnextchar[{\@@hcaption}{\@hcaption}}
{\begin{trivlist}\@hcaptype=0\centering\item[]}%
{\end{trivlist}}
{\begin{trivlist}\@hcaptype=1\centering\item[]}%
{\end{trivlist}}
\catcode`@=12